\documentclass{irmaart}

\usepackage{latexsym,amsmath,amssymb,amsfonts}  
\usepackage{color}
\usepackage{colordvi}
\usepackage{xspace}
\usepackage{axodraw}
\usepackage[all]{xy}            
\usepackage{graphicx}
\usepackage{pstricks}

\newcommand{\N}{\mathbb{N}}
\newcommand{\Z}{\mathbb{Z}}
\newcommand{\R}{\mathbb{R}}
\newcommand{\CC}{\mathbb{C}}
\newcommand{\K}{\mathbb{K}}

\renewcommand{\S}{\mathbb{S}}
\font\sevenrm=cmr7
\font\fiverm=cmr5
\long\def\ignore#1{}
\newcommand{\nc}{\newcommand}
\nc{\mop}[1]{\mathop{\hbox {\rm #1} }\nolimits}
\nc{\gmop}[1]{\mathop{\hbox {\bf #1} }\nolimits}
\nc{\smop}[1]{\mathop{\hbox {\sevenrm #1} }\nolimits}
\nc{\ssmop}[1]{\mathop{\hbox {\fiverm #1} }\nolimits}
\nc{\mopl}[1]{\mathop{\hbox {\rm #1} }\limits}
\nc{\smopl}[1]{\mathop{\hbox {\sevenrm #1} }\limits}
\nc{\ssmopl}[1]{\mathop{\hbox {\fiverm #1} }\limits}
\nc{\Cal}[1]{{\mathcal {#1}}}
\def\inj#1{\mathop{\hbox to #1 mm{$\lhook\joinrel$\rightarrowfill}}\limits}
\def\diagramme #1{\vskip 4mm \centerline {#1} \vskip 4mm}
\renewcommand{\d}{\mathrm{d}}
\newcommand{\BB}{\mathrm{B}}
\newcommand{\id}{\mathrm{id}}
\newcommand{\Diff}{\mathrm{Diff}}
\newcommand{\Gdif}{G^{\smop{dif}}}
\newcommand{\Gdifop}{G^{\smop{dif},\smop{op}}}
\newcommand{\Ginv}{G^{\smop{inv}}}

\newcommand{\GRT}{G_{\smop{RT}}}
\newcommand{\GFG}{G_{\smop{FG}}}
\newcommand{\ovGFG}{\overline{G}_{\smop{FG}}}
\newcommand{\Groups}{\mathrm{Groups}}
\newcommand{\HFdB}{\Cal H_{\smop{FdB}}}
\newcommand{\HFdBnc}{\Cal H_{\smop{FdB}}^{\smop{nc}}}
\newcommand{\HFdBncop}{\Cal H_{\smop{FdB}}^{\smop{nc},\smop{op}}}
\newcommand{\Hinv}{\Cal H_{\smop{inv}}}
\newcommand{\Hinvnc}{\Cal H_{\smop{inv}}^{\smop{nc}}}
\newcommand{\HRT}{\mathcal H_{\smop{RT}}} 
\newcommand{\HFG}{\mathcal H_{\smop{FG}}} 
\newcommand{\ovHFG}{\overline{\mathcal H}_{\smop{FG}}} 
\newcommand{\HPRT}{\mathcal H_{\smop{PRT}}}
\newcommand{\D}{\Delta}
\newcommand{\Dinv}{\Delta_{\smop{inv}}}
\newcommand{\Dinvstar}{\Delta_{\smop{inv}}^{*}}
\newcommand{\DFdB}{\Delta_{\smop{FdB}}}
\newcommand{\DFdBop}{\Delta_{\smop{FdB}}^{\smop{op}}}
\newcommand{\DFdBnc}{\Delta_{\smop{FdB}}^{\smop{nc}}}
\newcommand{\DRT}{\Delta_{\smop{RT}}}
\newcommand{\DFG}{\Delta_{\smop{FG}}}
\newcommand{\wtDRT}{\wt{\Delta}_{\smop{RT}}}
\newcommand{\Vect}{\mathrm{Vect}}
\newcommand{\Span}{\mathrm{Span}}

\newcommand{\Der}{\mathrm{Der}\,}

\newcommand{\Prim}{\mathrm{Prim}\,}
\newcommand{\Ind}{\mathrm{Ind}}
\renewcommand{\P}{\mathcal{P}}
\newcommand{\gLie}{\mathfrak{g}}
\newcommand{\g}{\gamma}

\newcommand{\Alg}{\mathrm{Alg}}
\newcommand{\CAlg}{\mathrm{CAlg}}

\newcommand{\Hom}{\operatorname{Hom}}
\def\un{\hbox{\bf 1}}
\newcommand{\wt}{\widetilde}
\def \restr#1{\mathstrut_{\textstyle |}\raise-6pt\hbox{$\scriptstyle #1$}}
\def \srestr#1{\mathstrut_{\scriptstyle |}\hbox to
  -1.5pt{}\raise-4pt\hbox{$\scriptscriptstyle #1$}}
\def\over{\slash}
\def\under{\backslash}
\numberwithin{equation}{section}
\newtheorem{thm}{Th\'eor\`eme}

\newtheorem{prop}[thm]{Proposition}

\newcommand{\shu}{\,
\setlength{\unitlength}{3pt}
\begin{picture}(0,0)
\put(0,0){\line(1,0){2}}
\put(0,0){\line(0,1){1}}
\put(1,0){\line(0,1){1}}
\put(2,0){\line(0,1){1}}
\end{picture}\ \ \, }

\def\operation{\,{\scalebox{0.6}{  
  \begin{picture}(148,136) (9,-15)
    \SetWidth{1}
    \SetColor{Black}
    \GBox(9,27)(138,46){0.882}
    \Line(16,46)(16,100)
    \Line(127,46)(127,100)
    \Line(67,100)(67,46)
    \Line(73,27)(73,2)
    \Text(68,32)[lb]{\Large{\Black{$a$}}}
    \Text(-70,-15)[lb]{\Large{\Black{Graphical representation of an $n$-ary operation}}}
    \Text(18,49)[lb]{\large{\Black{$1$}}}
    \Text(35,61)[lb]{\huge{\Black{$\cdots$}}}
     \Text(85,61)[lb]{\huge{\Black{$\cdots$}}}
    \Text(71,49)[lb]{\large{\Black{$i$}}}
    \Text(130,49)[lb]{\large{\Black{$n$}}}
  \end{picture}
}}\,}

\def\operade{\,{\scalebox{0.6}{  
  \begin{picture}(148,136) (9,-15)
    \SetWidth{1}
    \SetColor{Black}
    \GBox(35,71)(103,89){0.882}
    \GBox(9,27)(138,46){0.882}
    \Line(60,120)(60,89)
    \Line(44,120)(44,89)
    \Line(91,120)(91,89)
    \Line(77,120)(77,89)
    \Line(16,46)(16,120)
    \Line(113,46)(113,120)
    \Line(127,46)(127,120)
    \Line(67,70)(67,46)
    \Line(73,27)(73,2)
    \Text(63,75)[lb]{\Large{\Black{$b$}}}
    \Text(68,31)[lb]{\Large{\Black{$a$}}}
    \Text(-5,-15)[lb]{\Large{\Black{Partial composition $a\circ_i b$}}}
    \Text(18,49)[lb]{\large{\Black{$1$}}}
    \Text(18,105)[lb]{\large{\Black{$1$}}}
    \Text(20,61)[lb]{\large{\Black{$\cdots$}}}
    \Text(96,61)[lb]{\large{\Black{$\cdots$}}}
    \Text(71,49)[lb]{\large{\Black{$i$}}}
    \Text(130,105)[lb]{\large{\Black{$n+m-1$}}}
    \Text(130,49)[lb]{\large{\Black{$n$}}}
  \end{picture}
}}\,}

\def\asoperade{\,{\scalebox{0.6}{  
\begin{picture}(585,186) (104,-20)
    \SetWidth{1}
    \SetColor{Black}
    \GBox(165,109)(227,126){0.882}
    \GBox(134,65)(253,83){0.882}
    \GBox(104,17)(319,36){0.882}
    \Line(173,156)(173,126)
    \Line(182,156)(182,126)
    \Line(192,156)(192,126)
    \Line(219,156)(219,126)
    \Line(138,156)(138,84)
    \Line(237,156)(237,84)
    \Line(248,156)(248,84)
    \Line(108,155)(108,36)
    \Line(262,154)(262,35)
    \Line(305,156)(305,37)
    \Line(194,109)(194,83)
    \Line(194,64)(194,36)
    \Line(206,17)(206,-3)
    \GBox(474,18)(689,37){0.882}
    \GBox(165,109)(227,126){0.882}
    \GBox(165,109)(227,126){0.882}
    \GBox(492,64)(554,81){0.882}
    \GBox(613,64)(675,81){0.882}
    \Line(502,112)(502,82)
    \Line(516,112)(516,82)
    \Line(544,112)(544,82)
    \Line(630,111)(630,81)
    \Line(641,111)(641,81)
    \Line(667,112)(667,82)
    \Line(621,111)(621,81)
    \Line(522,63)(522,37)
    \Line(644,64)(644,36)
    \Line(581,18)(581,-1)
    \Line(478,112)(478,38)
    \Line(562,111)(562,37)
    \Line(606,111)(606,37)
    \Line(685,112)(685,38)
    \Text(202,24)[lb]{\Large{\Black{$a$}}}
    \Text(190,70)[lb]{\Large{\Black{$b$}}}
    \Text(192,115)[lb]{\Large{\Black{$c$}}}
    \Text(578,24)[lb]{\Large{\Black{$a$}}}
    \Text(520,69)[lb]{\Large{\Black{$b$}}}
    \Text(641,69)[lb]{\Large{\Black{$c$}}}
    \Text(145,-20)[lb]{\Large{\Black{Nested associativity}}}
    \Text(525,-16)[lb]{\Large{\Black{Disjoint associativity}}}
    \Text(112,43)[lb]{\large{\Black{$1$}}}
    \Text(197,44)[lb]{\large{\Black{$i$}}}
    \Text(308,44)[lb]{\large{\Black{$k$}}}
    \Text(691,44)[lb]{\large{\Black{$k$}}}
    \Text(111,67)[lb]{\large{\Black{$\cdots$}}}
    \Text(141,91)[lb]{\large{\Black{$1$}}}
     \Text(624,91)[lb]{\large{\Black{$1$}}}
    \Text(197,90)[lb]{\large{\Black{$j$}}}
    \Text(251,90)[lb]{\large{\Black{$l$}}}
     \Text(672,90)[lb]{\large{\Black{$m$}}}
    \Text(275,67)[lb]{\large{\Black{$\cdots$}}}
    \Text(482,44)[lb]{\large{\Black{$1$}}}
    \Text(526,45)[lb]{\large{\Black{$i$}}}
    \Text(650,44)[lb]{\large{\Black{$j$}}}
    \Text(505,88)[lb]{\large{\Black{$1$}}}
    \Text(547,88)[lb]{\large{\Black{$l$}}}
    \Text(519,104)[lb]{\large{\Black{$\cdots$}}}
    \Text(644,104)[lb]{\large{\Black{$\cdots$}}}
    \Text(194,150)[lb]{\large{\Black{$\cdots$}}}
    \Text(575,76)[lb]{\large{\Black{$\cdots$}}}
  \end{picture}
}}\,}

\def\gcoperade{\,{\scalebox{0.6}{
\begin{picture}(500,181) (102,-149)
    \SetWidth{1}
    \SetColor{Black}
    \GBox(102,-100)(602,-72){0.882}
    \GBox(132,-29)(198,-10){0.882}
    \GBox(230,-29)(313,-10){0.882}
    \GBox(395,-29)(550,-10){0.882}
    \Line(141,31)(141,-10)
    \Line(157,31)(157,-10)
    \Line(181,31)(181,-10)
    \Line(241,31)(241,-10)
    \Line(259,31)(259,-10)
    \Line(279,31)(279,-10)
    \Line(300,31)(300,-10)
    \Line(405,31)(405,-10)
    \Line(421,30)(421,-10)
    \Line(448,29)(448,-10)
    \Line(486,31)(486,-10)
    \Line(534,31)(534,-10)
    \Text(338,-91)[lb]{\Large{\Black{$a$}}}
    \Text(159,-24)[lb]{\Large{\Black{$b_1$}}}
    \Text(266,-24)[lb]{\Large{\Black{$b_2$}}}
    \Text(463,-24)[lb]{\Large{\Black{$b_n$}}}
    \Line(166,-29)(166,-72)
    \Line(270,-29)(270,-72)
    \Line(471,-29)(471,-72)
    \Text(335,-20)[lb]{\Large{\Black{$\cdots$}}}
    \Line(344,-100)(344,-123)
    \Text(209,-149)[lb]{\Large{\Black{Total composition $\gamma(a;b_1,b_2,\ldots,b_n)$}}}
  \end{picture}
}}\,}

\def\racine{{\scalebox{0.3}{ 
\begin{picture}(12,12)(38,-38)
\SetWidth{0.5} \SetColor{Black} \Vertex(45,-33){5.66}
\end{picture}}}}
  
 \def\arbrea{\,{\scalebox{0.15}{ 
  \begin{picture}(8,55) (370,-248)
    \SetWidth{2}
    \SetColor{Black}
    \Line(374,-244)(374,-200)
    \Vertex(374,-197){9}
    \Vertex(375,-245){12}
  \end{picture}
}}\,}

 \def\arbreba{\,{\scalebox{0.15}{ 
\begin{picture}(8,106) (370,-197)
    \SetWidth{2}
    \SetColor{Black}
    \Line(374,-193)(374,-149)
    \Vertex(374,-146){9}
    \Vertex(375,-194){12}
    \Line(374,-142)(374,-98)
    \Vertex(374,-95){9}
  \end{picture}
}}\,}

 \def\arbrebb{\,{\scalebox{0.15}{ 
  \begin{picture}(48,48) (349,-255)
    \SetWidth{2}
    \SetColor{Black}
    \Vertex(375,-252){12}
    \Line(376,-250)(395,-215)
    \Line(373,-251)(354,-214)
    \Vertex(353,-211){9}
    \Vertex(395,-213){9}
  \end{picture}
}}}

\def\arbreca{\,{\scalebox{0.15}{
\begin{picture}(8,156) (370,-147)
    \SetWidth{2}
    \SetColor{Black}
    \Line(374,-143)(374,-99)
    \Vertex(374,-96){9}
    \Vertex(375,-144){12}
    \Line(374,-92)(374,-48)
    \Vertex(374,-45){9}
    \Line(374,-42)(374,2)
    \Vertex(374,5){9}
  \end{picture}
}}\,}

\def\arbrecb{\,{\scalebox{0.15}{
\begin{picture}(48,94) (349,-255)
\SetWidth{2}
\SetColor{Black}
\Line(376,-204)(395,-169)
\Line(373,-205)(354,-168)
\Vertex(353,-165){9}
\Vertex(395,-167){9}
\Vertex(374,-205){9}
\Line(374,-246)(374,-209)
\Vertex(374,-252){12}
\end{picture}}}\,}

\def\arbrecc{\,{\scalebox{0.15}{
 \begin{picture}(48,98) (349,-205)
    \SetWidth{2}
    \SetColor{Black}
    \Vertex(375,-202){12}
    \Line(376,-200)(395,-165)
    \Line(373,-201)(354,-164)
    \Vertex(353,-161){9}
    \Vertex(395,-163){9}
    \Line(353,-160)(353,-113)
    \Vertex(353,-111){9}
  \end{picture}
}}\,}

\def\arbrecd{\,{\scalebox{0.15}{
\begin{picture}(48,52) (349,-251)
    \SetWidth{2}
    \SetColor{Black}
    \Vertex(375,-248){12}
    \Line(376,-246)(395,-211)
    \Line(373,-247)(354,-210)
    \Vertex(353,-207){9}
    \Vertex(395,-209){9}
    \Line(375,-247)(375,-206)
    \Vertex(376,-203){9}
  \end{picture}
 }}\,}

\def\arbreda{\,{\scalebox{0.15}{
\begin{picture}(8,204) (370,-99)
    \SetWidth{2}
    \SetColor{Black}
    \Line(374,-95)(374,-51)
    \Vertex(374,-48){9}
    \Vertex(375,-96){12}
    \Line(374,-44)(374,0)
    \Vertex(374,3){9}
    \Line(374,6)(374,50)
    \Vertex(374,53){9}
    \Line(374,53)(374,98)
    \Vertex(374,101){9}
  \end{picture}
}}\,}

\def\arbredb{\,{\scalebox{0.15}{
\begin{picture}(48,135) (349,-255)
    \SetWidth{2}
    \SetColor{Black}
    \Line(376,-163)(395,-128)
    \Line(373,-164)(354,-127)
    \Vertex(353,-124){9}
    \Vertex(395,-126){9}
    \Vertex(374,-164){9}
    \Line(374,-205)(374,-168)
    \Vertex(374,-207){9}
    \Line(374,-248)(374,-211)
    \Vertex(374,-252){12}
  \end{picture}
}}\,}

\def\arbredc{\,{\scalebox{0.15}{
 \begin{picture}(48,150) (349,-205)
    \SetWidth{2}
    \SetColor{Black}
    \Line(376,-148)(395,-113)
    \Line(373,-149)(354,-112)
    \Vertex(353,-109){9}
    \Vertex(395,-111){9}
    \Line(353,-108)(353,-61)
    \Vertex(353,-59){9}
    \Line(374,-200)(374,-153)
    \Vertex(374,-149){9}
    \Vertex(374,-202){12}
  \end{picture}
}}\,}

\def\arbredd{\,{\scalebox{0.15}{
 \begin{picture}(48,99) (349,-251)
    \SetWidth{2}
    \SetColor{Black}
    \Line(376,-199)(395,-164)
    \Line(373,-200)(354,-163)
    \Vertex(353,-160){9}
    \Vertex(395,-162){9}
    \Vertex(376,-156){9}
    \Vertex(376,-248){12}
    \Line(375,-245)(375,-204)
    \Line(375,-200)(375,-159)
    \Vertex(375,-201){9}
  \end{picture}
}}\,}

\def\arbrede{\,{\scalebox{0.15}{
 \begin{picture}(48,153) (349,-150)
    \SetWidth{2}
    \SetColor{Black}
    \Vertex(375,-147){12}
    \Line(376,-145)(395,-110)
    \Line(373,-146)(354,-109)
    \Vertex(353,-106){9}
    \Vertex(395,-108){9}
    \Line(353,-105)(353,-58)
    \Vertex(353,-56){9}
    \Line(353,-52)(353,-5)
    \Vertex(353,-1){9}
  \end{picture}
}}\,}

\def\arbredf{\,{\scalebox{0.15}{
\begin{picture}(48,98) (349,-205)
    \SetWidth{2}
    \SetColor{Black}
    \Vertex(375,-202){12}
    \Line(376,-200)(395,-165)
    \Line(373,-201)(354,-164)
    \Vertex(353,-161){9}
    \Vertex(395,-163){9}
    \Line(353,-160)(353,-113)
    \Vertex(353,-111){9}
    \Line(395,-159)(395,-112)
    \Vertex(395,-111){9}
  \end{picture}
}}\,}

\def\arbredz{\,{\scalebox{0.15}{
  \begin{picture}(68,88) (329,-215)
    \SetWidth{2}
    \SetColor{Black}
    \Vertex(375,-212){12}
    \Line(376,-210)(395,-175)
    \Line(373,-211)(354,-174)
    \Vertex(353,-171){9}
    \Vertex(395,-173){9}
    \Line(351,-168)(332,-131)
    \Line(355,-168)(374,-133)
    \Vertex(333,-131){9}
    \Vertex(374,-131){9}
  \end{picture}
}}\,}

\def\arbredg{\,{\scalebox{0.15}{
\begin{picture}(48,98) (349,-205)
    \SetWidth{2}
    \SetColor{Black}
    \Vertex(375,-202){12}
    \Line(376,-200)(395,-165)
    \Line(373,-201)(354,-164)
    \Vertex(353,-161){9}
    \Vertex(395,-163){9}
    \Line(375,-201)(375,-160)
    \Vertex(376,-157){9}
    \Vertex(376,-111){9}
    \Line(375,-155)(375,-114)
  \end{picture}
}}\,}

\def\arbredh{\,{\scalebox{0.15}{
 \begin{picture}(90,46) (330,-257)
    \SetWidth{2}
    \SetColor{Black}
    \Vertex(375,-254){12}
    \Line(376,-252)(395,-217)
    \Vertex(395,-215){9}
    \Line(374,-254)(335,-226)
    \Vertex(334,-224){9}
    \Line(375,-252)(356,-215)
    \Vertex(355,-215){9}
    \Line(374,-255)(417,-227)
    \Vertex(418,-225){9}
  \end{picture}
}}\,}

\newcommand{\treeO}{
\setlength{\unitlength}{3pt}
\psset{unit=3pt}
\psset{runit=2pt}
\psset{linewidth=0.2}
\begin{pspicture}(0,0)(2.5,2)
\psline(1,-1)(1,1.5)
\end{pspicture}}

\newcommand{\treeA}{
\setlength{\unitlength}{3pt}
\psset{unit=3pt}
\psset{runit=2pt}
\psset{linewidth=0.2}
\begin{pspicture}(0,0)(2.5,2)
\psline(1,-1)(1,.5)
\psline(1,.5)(0,1.5)
\psline(1,.5)(2,1.5)
\end{pspicture}}

\newcommand\treeAB{
\setlength{\unitlength}{3pt}
\psset{unit=3pt}
\psset{runit=2pt}
\psset{linewidth=0.2}
\begin{pspicture}(0,0)(5,3)
\psline(3,-1)(3,.5)
\psline(3,.5)(1,2.5)
\psline(3,.5)(4,1.5)
\psline(2,1.5)(3,2.5)
\end{pspicture}}

\newcommand\treeBA{
\setlength{\unitlength}{3pt}
\psset{unit=3pt}
\psset{runit=2pt}
\psset{linewidth=0.2}
\begin{pspicture}(0,0)(5,3)
\psline(2,-1)(2,.5)
\psline(2,.5)(4,2.5)
\psline(2,.5)(1,1.5)
\psline(3,1.5)(2,2.5)
\end{pspicture}}

\newcommand\treeABC{
\setlength{\unitlength}{3pt}
\psset{unit=3pt}
\psset{runit=2pt}
\psset{linewidth=0.2}
\begin{pspicture}(0,0)(5,4.5)
\psline(3,-1)(3,.5)
\psline(3,.5)(0,3.5)
\psline(3,.5)(4,1.5)
\psline(2,1.5)(3,2.5)
\psline(1,2.5)(2,3.5)
\end{pspicture}}

\newcommand\treeBAC{
\setlength{\unitlength}{3pt}
\psset{unit=3pt}
\psset{runit=2pt}
\psset{linewidth=0.2}
\begin{pspicture}(0,0)(5,4.5)
\psline(3,-1)(3,.5)
\psline(3,.5)(1,2.5)
\psline(3,.5)(4,1.5)
\psline(2,1.5)(4,3.5)
\psline(3,2.5)(2,3.5)
\end{pspicture}}

\newcommand\treeACA{
\setlength{\unitlength}{3pt}
\psset{unit=3pt}
\psset{runit=2pt}
\psset{linewidth=0.2}
\begin{pspicture}(0,0)(6,3.5)
\psline(3,-1)(3,.5)
\psline(3,.5)(0.5,3)
\psline(3,.5)(5.5,3)
\psline(1.5,2)(2.5,3)
\psline(4.5,2)(3.5,3)
\end{pspicture}}

\newcommand\treeCAB{
\setlength{\unitlength}{3pt}
\psset{unit=3pt}
\psset{runit=2pt}
\psset{linewidth=0.2}
\begin{pspicture}(0,0)(5,4.5)
\psline(2,-1)(2,.5)
\psline(2,.5)(4,2.5)
\psline(2,.5)(1,1.5)
\psline(3,1.5)(1,3.5)
\psline(2,2.5)(3,3.5)
\end{pspicture}}

\newcommand\treeCBA{
\setlength{\unitlength}{3pt}
\psset{unit=3pt}
\psset{runit=2pt}
\psset{linewidth=0.2}
\begin{pspicture}(0,0)(5,4.5)
\psline(2,-1)(2,.5)
\psline(2,.5)(5,3.5)
\psline(2,.5)(1,1.5)
\psline(3,1.5)(2,2.5)
\psline(4,2.5)(3,3.5)
\end{pspicture}}

\newcommand{\lgraft}[2]
{\setlength{\unitlength}{4pt}
\psset{unit=5pt}
\psset{runit=4pt}
\psset{linewidth=.1}
\begin{pspicture}(0,0)(4,4)
\psline(2.1,.9)(1.4,1.6)
\put(2.6,0){$#1$}
\put(0.5,2.6){$#2$}
\end{pspicture}}

\newcommand{\rgraft}[2]
{\setlength{\unitlength}{4pt}
\psset{unit=5pt}
\psset{runit=4pt}
\psset{linewidth=.1}
\begin{pspicture}(0,0)(4,4)
\psline(1.5,.9)(2.2,1.6)
\put(.4,0){$#1$}
\put(3,2.6){$#2$}
\end{pspicture}}

\begin{document}

\title{Five interpretations of Fa\`a di Bruno's formula}
\author{Alessandra Frabetti (*), Dominique Manchon (**) \\
\address{
(*) Universit\'e de Lyon, Universit\'e Lyon 1, CNRS, \\  
UMR 5208 Institut Camille Jordan, \\ 
B\^atiment du Doyen Jean Braconnier, \\
43, blvd du 11 novembre 1918, F-69622 Villeurbanne Cedex, France. \\
email:{ \tt frabetti@math.univ-lyon1.fr}\\
(**) Universit\'e Blaise Pascal, CNRS, \\  
UMR 6620, Laboratoire de Math\'ematiques, \\ 
BP 80026, \\
F-63171 Aubi\`ere Cedex, France. \\
email:{ \tt manchon@math.univ-bpclermont.fr}}}

\date{\today}
\maketitle

\begin{center}
Dedicated to Jean-Louis Loday
\end{center}
\vskip 8mm

\begin{abstract}
In these lectures we present five interpretations of the Fa\`a di Bruno
formula which computes the $n$-th derivative of the composition of two 
functions of one variable: in terms of groups, Lie algebras and Hopf algebras, 
in combinatorics and within operads.
\end{abstract}
\bigskip 

\medskip 



{\small \tableofcontents} 
\bigskip\bigskip 

\section{Introduction}
In these lectures we present five interpretations of the Fa\`a di Bruno
formula which computes the $n$-th derivative of the composition of two 
functions of one variable.
\begin{enumerate}
\item
This formula tells explicitly how to compute the composition law in the group
of formal diffeomorphisms in one variable. It is therefore related to the 
Lagrange inversion formula, which computes the inverse of a formal 
diffeomorphism in this group. 
\item
In terms of Hopf algebras, it gives the coproduct of the so-called 
Fa\`a di Bruno Hopf algebra, which can be seen as the coordinate ring 
of the previous group. 
It is one of the renormalization Hopf algebras which appeared recently in 
quantum field theory, and being related to the composition of formal series 
it is expected to appear, in its one-variable or several-variables form, 
in any perturbative theory. 
The Fa\`a di Bruno Hopf algebra allows a non-commutative lift, but unlike 
what happens for the group of invertible formal series endowed with the 
product law, this non-commutative Hopf algebra does not represent a 
pro-algebraic group as a functor on non-commutative associative algebras. 
Its nature is then still an open question.
\item
The Lie algebra of the formal diffeomorphisms group is a well-known 
subalgebra of the Witt (and therefore also of the Virasoro) Lie algebra: 
the Fa\`a di Bruno formula is therefore related to the well-known Lie 
bracket on these Lie algebras.  The Fa\`a di Bruno Hopf algebra is obtained from the  Lie algebra of the formal diffeomorphisms group by considering the graded dual of its envoloping algebra (Cartier-Milnor-Moore correspondence).
\item
In combinatorics, the Fa\`a di Bruno Hopf algebra is an important example of
incidence Hopf algebra \cite{Schmitt}. 
It is also a right-sided combinatorial 
Hopf algebra \cite{LR2}, and therefore its associated Lie bracket is the commutator 
of a brace product. 
It would be interesting to extend this brace product to the Witt algebra, 
which is the Lie algebra of vector fields on the circle. 
\item
Finally, an interpretation of the Fa\`a di Bruno formula in operadic terms 
is also possible. We recall in Section \ref{sect:operad} the definition of a prounipotent group (and its corresponding pronilpotent Lie algebra) associated to a large class of linear operad \cite{Chapoton, VanDerLaan}, the group law being related to the operadic 
composition. 
Then, the group of formal diffeomorphisms (hence the Fa\`a di Bruno formula) is the group associated to the operad $Assoc$ governing associative algebras. 
\end{enumerate}
\textbf{Acknowledgements} We thank the referees for their careful reading and their pertinent suggestions which greatly helped us to improve the text.
\section{Fa\`a di Bruno's formula}
\subsection{Francesco Fa\`a di Bruno}

The ``Cavaliere'' Francesco Fa\`a di Bruno was born on March 9, 1825, 
in Alessandria (Italy), and was the youngest of twelve children. 
He died on March 27, 1888, in Torino (Italy). 
He joined the Military Academy of Torino at the age of 15, took part 
in the first Italian independence war in 1848, and then went to Paris 
for two years, during which he studied mathematics with Augustin-Louis 
Cauchy and with the astronomer Urbain Jean Joseph Le Verrier. 
Back in Torino, he left the army in 1853. He gave mathematics lectures 
at Torino University and at the Military Academy, and was appointed 
Professor in 1876. 

Besides his official teaching and research job, he carried out a lot of 
scientific and human activities such as mechanics (he invented, 
among other things, an electric alarm clock for his blind sister), 
music, and, above all, social and philanthropic activities inspired by 
his Christian faith. 
He became a catholic priest in 1876, and was beatified by John-Paul II 
in 1988 for his courageous fight against the hard living conditions 
of female workers in Torino in the 19th century. The interested reader can consult the biography by L. Giacardi \cite{Giacardi} (in Italian).

\subsection{Fa\`a di Bruno's and Lagrange's formulae}

Let us consider the set of smooth functions in one variable, 
endowed with the composition. If $f$ and $g$ are smooth functions of $t$, 
their composition is the smooth function $f\circ g$ defined by 
$(f\circ g)(t)=f\big(g(t)\big)$. 
The Fa\`a di Bruno formula tells how to compute the $n$-th derivative 
of the composite function $f\circ g$ in terms of the derivatives 
of $f$ and $g$. This famous formula was published by F.~Fa\`a di Bruno 
in the dense two-page paper ``{\em Sullo sviluppo delle funzioni}'' 
in 1855 \cite{FaadiBruno}, and says that 
\begin{align}
\label{FdB-formula}
\frac{\d^n}{\d t^n}f\big(g(t)\big) 
&= \sum_{m=1}^n\sum \frac{n!}{k_1! k_2!\cdots k_n!}\ 
f^{(m)}\big(g(t)\big)\ \Big(\frac{g'(t)}{1!}\Big)^{k_1}\ 
\Big(\frac{g''(t)}{2!}\Big)^{k_2}\cdots \Big(\frac{g^{(n)}(t)}{n!}\Big)^{k_n}, 
\end{align}
where the second sum is taken over the non-negative integers $k_1,...,k_n$ 
such that  $k_1+k_2+\cdots +k_n=m$ and $k_1+2k_2+\cdots +nk_n=n$. 
For instance, for $n=3$, there are three possible sequences of such numbers: 
$(k_1,k_2,k_3)=(0,0,1)$, $(1,1,0)$ and $(3,0,0)$. Then 
\begin{align*}
\frac{\d^3}{\d t^3}f\big(g(t)\big) 
& = \frac{3!}{1!} f'\big(g(t)\big)\ \frac{g'''(t)}{3!} 
+ \frac{3!}{1!} f''\big(g(t)\big)\ \frac{g'(t)}{1!} \frac{g''(t)}{2!} 
+ \frac{3!}{3!} f'''\big(g(t)\big)\ \frac{g'(t)^3}{(1!)^3} \\
& = f'\big(g(t)\big)\ g'''(t) + 3 f''\big(g(t)\big)\ g'(t) g''(t) 
+ f'''\big(g(t)\big)\ g'(t)^3. 
\end{align*}

The most direct approach at this stage is to prove  \eqref{FdB-formula} directly by induction on $n$. Although it does not show particular difficulties, it is rather cumbersome. A very elegant alternative proof of \eqref{FdB-formula} can be found in \cite{FLM} (Proposition 8.3.4 therein), which  heuristically goes as follows: we can extend any smooth function $f$ to a smooth functional $\wt f$ from $\mathbb R[[y]]$ to $\mathbb R[[y]]$ by Taylor expansion:
\begin{eqnarray*}
\wt f\left(\sum_{k\ge 0}t_ky^k\right)&=&\wt f\left(t_0+\sum_{k\ge 1}t_ky^k\right)\\
&:=&\sum_{n\ge 0}\frac{\left(\sum_{k\ge 1}t_ky^k\right)^n}{n!}\partial^nf(t_0).
\end{eqnarray*}
Functionals like $T\mapsto \wt f(T)$ can be seen as smooth functions depending on the infinite set of variables $\{t_0,t_1,t_2,\ldots\}$, with $T=t_0+t_1y+t_2y^2+\cdots$. Set now $\partial_k:=\displaystyle\frac{d}{dt_k}$ for short and collect all the powers $\partial_0^n$ in the formal translation operator $e^{y\partial_0}=\sum_{k\ge 0}\frac{y^k}{k!}\partial_0^k$. The crucial properties:
\begin{eqnarray}
\wt{f\circ g}&=&\wt f\circ \wt g,\label{property1}\\
e^{y\partial_0}(\wt f\circ \wt g)(t_0)&=&\wt f\circ(e^{y\partial_0}\wt g)(t_0)=\wt{f\circ g}(t_0+y)\label{property2}
\end{eqnarray}
together yield:
\begin{eqnarray}
e^{y\partial_0}(\wt{f\circ g})(t_0)&=&\wt f\big(\wt g(t_0+y)\big)\notag\\
&=&\wt f\left(g(t_0)+\sum_{k\ge 1}\frac{g^{(k)}(t_0)}{k!}y^k\right)\notag\\
&=&e^{\left(\sum_{k\ge 1}\frac{g^{(k)}(t_0)}{k!}y^k\right)\frac{d}{dt_0}}\wt f\big(g(t_0)\big)\label{gen-fdb}.
\end{eqnarray}
Formula \eqref{FdB-formula} is then obtained by taking $n!$ times the coefficient of $y^n$ in \eqref{gen-fdb} above. Formula \eqref{gen-fdb} is valid in greater generality, replacing $\partial_0$ with any derivation $X$, with $g^{(k)}$ now standing for $X^k(g)$. \\

In order to make the proof above completely rigorous, we have to prove properties \eqref{property1} and \eqref{property2}. A proof of these can be found in \cite{FLM}, based uniquely on formal-variable calculus (see \cite[Sections 2.1, 2.2, 8.1--8.3]{FLM}). One far-reaching application of this formal-variable calculus is the associativity principle for lattice vertex operators (\cite{FLM}, Equation (8.4.32), Theorem 8.4.2, Theorem 8.8.9 and Proposition 8.10.5). We won't go further into this direction, referring the interested reader to \cite{FLM, LL} for an introduction to the theory of vertex operator algebras. See also \cite{Kac} for a different approach to this vast topic.\\

Our interpretation of the formal variable argument alluded to above is the following: for any Taylor expansion functional $F=\wt f$ as above, where $f:\mathbb R\to\mathbb R$ is a smooth function, and for any $S,T\in\mathbb R[[y]]$ we have:
\begin{equation}\label{translation}
F(T+yS)=e^{yS\partial_0}F(T).
\end{equation}
The case $T=t_0\in\mathbb R$ is a simple rephrasing of the definition, and the case for general $T$ follows easily. Taking the linear part with respect to $S$ in \eqref{translation} successively for $S=1,y,y^2,\ldots$, we get an explicit expression for the differential $DF$ of $F$ at $T$ in terms of $DF(T)[\partial_0]$. Namely, for any $k\ge 1$:
\begin{equation}\label{diffk}
DF(T)[\partial_k]=y^k DF(T)[\partial_0].
\end{equation}
We can identify $\mathbb R[[y]]$ with its tangent space at each point $T$, via $\partial_k\simeq y^k$. We have then a coordinate-free formulation of \eqref{diffk}:
\begin{equation}\label{diffkbis}
DF(T)[S]=S.DF(T)[\partial_0]
\end{equation}
for any $S\in\mathbb R[[y]]$. Conversely, any smooth functional $F:\mathbb R[[y]]\to\mathbb R[[y]]$ which satisfies \eqref{diffkbis} also satisfies the following identities involving higher-order differentials:
\begin{equation}\label{diffkter}
D^rF(T)[S_1,\ldots, S_r]=S_1\cdots S_r.D^rF(T)[\partial_{0},\ldots,\partial_{0}].
\end{equation}
Writing down the Taylor formula for the one-variable function $h\mapsto F(T+hyS)$ we see that \eqref{diffk} implies \eqref{translation}.  Hence any $F$ verifying \eqref{diffk} can be written as $F=\wt f$, with $\wt f(t_0)=F(t_0)$ for any $t_0\in\mathbb R$. Now consider two functionals $F=\wt f$ and $G=\wt g$. We claim that $F\circ G$ verifies \eqref{diffk}: indeed, 
\begin{eqnarray*}
D(F\circ G)(T)(\partial_k)&=&DF\big(G(T)\big)\circ DG(T)[\partial_k]\\
&=&DF\big(G(T)\big)\circ y^kDG(T)[\partial_0]\\
&=&y^kDF\big(G(T)\big)\circ DG(T)[\partial_0]\\
&=&y^kD(F\circ G)(T)[\partial_0].
\end{eqnarray*}
Hence there exists $h\in C^\infty(\mathbb R)$ such that $F\circ G=\wt h$, and evaluating it at $t_0\in\mathbb R$ yields $h=f\circ g$, which proves \eqref{property1}. Property \eqref{property2} follows immediately:
\begin{eqnarray*}
\wt f\circ e^{y\partial_0}\wt g(t_0)&=&\wt f\circ e^{\partial_1}\wt g(t_0)\\
&=&\wt f\circ\wt g(t_0+y)\\
&=&e^{\partial_1}(\wt f\circ\wt g)(t_0)\\
&=&e^{y\partial_0}(\wt f\circ\wt g)(t_0).
\end{eqnarray*}
Surprisingly enough, F. Fa\`a di Bruno wrote this formula almost 
a century after a similar one was established: the equally famous 
Lagrange inversion formula of 1770 \cite{Lagrange}, which computes the 
compositional inverse of a smooth function. 
If $u=f(t)$ and $t=g(u)$ are two smooth functions, inverse to each other, with $f(0)=g(0)=0$ and $f'(0)\not =0$,
then Lagrange's inversion formula says that, for $u$ in a neighborhood of zero, 
\begin{align}
\label{Lagrange-formula}
g(u)= \sum_{n=1}^N \frac{1}{n!}\ 
\frac{\d^{n-1}}{\d t^{n-1}}\left(\frac{f(t)}{t}\right)^{-n}\Big|_{t=0}\ u^n+O(u^{N+1})
\end{align}
for any $N\ge 0$. For instance, for $u=f(t)=t\,e^t$, this formula allows us to compute 
$t$ as 
\begin{align}\label{lambert}
t=g(u)& =\sum_{n=1}^\infty\ (-1)^{n-1}\ \frac{n^{n-1}}{n!}\ u^n,
\end{align} 
because
\begin{align*}
\frac{\d^{n-1}}{\d t^{n-1}}\left(\frac{f(t)}{t}\right)^{-n}\Big|_{t=0} 
= \frac{\d^{n-1}(e^{-nt})}{\d t^{n-1}}\Big|_{t=0} = (-n)^{n-1}. 
\end{align*}
The series \eqref{lambert} defines Lambert's W-function on its domain of convergence, with radius of 
convergence equal to $1/e$ (see e.g. \cite{T}). Similarly to the logarithm, this function 
admits a multivalued analytic extension to the whole complex plane 
with $1/e$ removed. 

In the two-page paper \cite{FaadiBruno}, Fa\`a di Bruno also expresses 
$\frac{\d^n}{\d t^n}f\big(g(t)\big)$ as the determinant of the following 
$n\times n$ matrix:
\begin{align*}
\left(\begin{array}{cccccccc}
f g' & {\scriptsize (n-1)} f g'' &\binom{n-1}{2} f g''' & 
\binom{n-1}{3} f g^{(4)} & \cdots & 
\binom{n-1}{n-3} f g^{(n-2)} & (n-1) f g^{(n-1)} & f g^{(n)} 
\\ 
-1 & f g' & (n-2) f g'' & \binom{n-2}{2} f g''' & \cdots & 
\binom{n-2}{n-4} f g^{(n-3)} & (n-2) f g^{(n-2)} & f g^{(n-1)} 
\\ 
0 & -1 & f g' & (n-3) f g'' & \cdots & 
\binom{n-3}{n-5} f g^{(n-4)} & (n-3) f g^{(n-3)} & f g^{(n-2)} 
\\ 
\vdots & \vdots & \vdots & \vdots & \cdots & \vdots & \vdots & \vdots 
\\ 
0 & 0 & 0 & 0 & \cdots & -1 & f g' & f g''
\\ 
0 & 0 & 0 & 0 & \cdots & 0 & -1 & f g' 
\end{array}\right)
\end{align*}
Unlike the Fa\`a di Bruno formula (\ref{FdB-formula}), this beautiful 
matrix formula seemingly did not get any echo among mathematicians, 
nor any further development.
\subsection{Some points of history}
It has been pointed out that Formula \eqref{FdB-formula} was certainly known 
before Fa\`a di Bruno's paper \cite{FaadiBruno}: quoting W.~P.~Johnson 
\cite{Johnson}, several formulas of the same type were established in the 
middle of the 19th century, by Hoppe (1845), Meyer (1847), J.F.C. Tiburce 
Abadie \textsl{alias} T.A. (1850) and Scott (1861). 
Moreover, as highlighted by A.~D.~D.~Craik \cite{Craik}, formula 
\eqref{FdB-formula} appears nearly as such, as early as 1800, in the 
remarkable book of Louis Fran\c cois Antoine Arbogast (1759-1803), professor 
of Mathematics in Strasbourg \cite[``Remarques'', pp. 43-44]{Arbogast}. 
This work on differential calculus for composite functions inspired in turn 
several French and British mathematicians of the first half of the 19th 
century, e.g. \cite{Knight, Lacroix}. To sum up, although we won't change 
a well-established terminology, let us cite the following sentence from 
the conclusion of Craik's historical survey \cite{Craik}:
\begin{quote}
\textsl{``Fa\`a di Bruno's formula'' was first stated by Arbogast in 1800, 
and it might as appropriately be named after one of the ten or more authors 
who obtained versions of it before Fa\`a di Bruno. Only the determinantal 
formulation of it ought to be called ``Fa\`a di Bruno's formula''}.
\end{quote}

In the middle of the 20th century the Fa\`a di Bruno formula is enriched 
with a combinatorial interpretation. In 1946, J. Riordan recognizes the 
Bell polynomials \cite{Bell}:
$$
\BB_{n,m}\big(x_1,x_2,...,x_l) = \sum_{{k_1+k_2+\cdots+ k_l=m,\atop k_1+2k_2+\cdots +lk_l=n}} \frac{n!}{k_1!\cdots k_l!}\ 
\Big(\frac{x_{1}}{1!}\Big)^{k_1}\cdots \Big(\frac{x_{l}}{l!}\Big)^{k_l}, 
\qquad l=n-m+1
$$ 
and rewrites Fa\`a di Bruno's formula as \cite{Riordan}:
\begin{equation}\label{eq:riordan}
\frac{\d^n}{\d t^n}f\big(g(t)\big) 
= \sum_{m=1}^n \ f^{(m)}\big(g(t)\big)\
\BB_{n,m}\Big(g'(t),g''(t),\ldots,g^{(n-m+1)}(t)\Big). 
\end{equation}
In 1965, R.~Frucht and G-C.~Rota \cite{FruchtRota} propose a purely 
combinatorial formulation of Fa\`a di Bruno's formula: 
$$
\frac{\d^n}{\d t^n}f\big(g(t)\big) 
= \sum_{m=1}^n \sum \ f^{(m)}\big(g(t)\big)\ 
g'(t)^{k_1} g''(t)^{k_2}\cdots g^{(n)}(t)^{k_n}, 
$$
where the second sum is taken over the partitions of $\{1,...,n\}$ 
with $m$ blocks having $k_i$ blocks with $i$ elements. 
For instance, for $n=3$, there are 5 partitions,  
$\{1,2,3\}$, $\{1\}\{2,3\}$, $\{2\}\{1,3\}$, $\{3\}\{1,2\}$ and 
$\{1\}\{2\}\{3\}$: 
\begin{itemize}
\item
$\{1,2,3\}$ has a single block with 3 elements, therefore to this partition 
corresponds the sequence $(k_1,k_2,k_3)= (0,0,1)$; 
\item
$\{1\}\{2,3\}$, $\{2\}\{1,3\}$ and $\{3\}\{1,2\}$ have two blocks, one with 
1 element and the other one with 2, therefore there are three sequences 
$(k_1,k_2,k_3)=(1,1,0)$; 
\item
$\{1\}\{2\}\{3\}$ has three blocks with 1 element each, therefore there 
is a corresponding sequence $(k_1,k_2,k_3)= (3,0,0)$. 
\end{itemize}
The number of resulting sequences matches with the coefficient in
Fa\`a di Bruno's formula. The proof of the Frucht-Rota formula starting from Riordan's formula resides in the following combinatorial interpretation of the Bell polynomials: the coefficient of $x_1^{j_1}\cdots x_l^{j_l}$ in $B_{n,m}(x_1,\ldots x_l)$ (with $l=n-m+1$) is equal to the number of partitions of a set of cardinality $n$ with $j_r$ blocks of size $r$, $r\in\{1,\ldots ,l\}$ (and therefore $m$ blocks altogether).

\section{Group interpretation}
\subsection{Formal diffeomorphisms}\label{fd}

The composition of functions of one variable is an associative operation, 
for which the identity $\id(t)=t$ is the unit. 
Therefore the set of smooth functions which have a compositional inverse 
forms a group, called the group of smooth diffeomorphisms on the real line, 
and denoted by $\Diff(\R)$. Groups of diffeomorphisms on a manifold are 
infinite dimensional groups, not locally compact, with many connected 
components classified by the mapping class group. Even for the simplest ones, 
like $\Diff(\R)$ or $\Diff(\S^1)$, very little is known.

Up to composition by a translation, a diffeomorphism on the real line 
fixes the origin. Up to a scalar factor, one can also assume that a 
diffeomorphism is tangent to the identity at the fixed point. 
In other words, the group of diffeomorphisms of the real line is the 
semi-direct product of the group of diffeomorphisms tangent to the identity 
by the group of affine transformations.

Thinking in terms of Taylor expansions, we will rather consider the group of 
\textsl{formal} diffeomorphisms tangent to the identity:
\begin{equation}
\Gdif(\K)=\Big\{ f(t)= \sum f_n t^n\in \K[[t]],\ f_0=0\hbox{ and }f_1=1 \Big\},
\end{equation}
whose definition makes sense for any field $\K$ of characteristic zero. 
Fa\`a di Bruno's formula then expresses the $n$-th coefficient of the series 
$f\circ g$ in terms of the coefficients of $f$ and $g$:
\begin{equation}\label{eq:comp}
(f\circ g)_n=\frac{1}{n!}(f\circ g)^{(n)}(0)
= \sum_{m=1}^n\sum \ \dfrac{m!}{k_1! k_2!\cdots k_n!}\quad 
f_m\, g_1^{k_1}\cdots g_n^{k_n},
\end{equation}
where the second sum runs over $n$-uples $(k_1,\ldots,k_n)$ of non-negative 
integers such that $k_1+k_2+\cdots +k_n=m$ and $k_1+2k_2+\cdots +nk_n=n$. 
In terms of Bell's polynomial, this coefficient becomes:
\begin{equation}\label{eq:comp-Bell}
(f\circ g)_n=
\sum_{m=1}^n\dfrac{m!}{n!}\quad f_m\, B_{n,m}(g_1,2!g_2,...,(n-m+1)! g_{n-m+1}). 
\end{equation}
Formula \eqref{eq:comp} can be directly checked by computing the coefficient 
of $t^n$ in the composition of the two formal series:
\begin{equation*}
(f\circ g)(t)=\sum_{m\ge 1}f_m\Big(\sum_{k\ge 1}g_kt^k\Big)^m,
\end{equation*}
expanding each power by means of the multinomial formula. The Fa\`a di Bruno 
formula \eqref{FdB-formula} is then directly derived from \eqref{eq:comp}.
\bigskip

This group admits finite dimensional representations, none of which is faithful.
Its smallest faithful representation is infinite dimensional, and presents 
$\Gdif(\K)$ as the group of infinite matrices of the form 
$$
M(g)=\left(\begin{array}{cccccc}
g_1 & g_2 & g_3 & g_4 & g_5 &\cdots \\ 
0 & g_1 & 2g_2 & g_2^2+2g_3 & 2g_4+2g_2g_3 &\cdots \\ 
0 & 0 & g_1 & 3g_2 & 3g_3+3g_2^2&\cdots \\ 
0 & 0 & 0 & g_1 & 4g_2&\cdots \\
0 & 0 & 0 & 0 & g_1&\cdots \\
\vdots & & &\cdots & & \ddots 
\end{array}\right),
$$
with $g_1=1$ here. In view of formula \eqref{eq:comp-Bell}, the coefficients $(M(g)_{ij})_{i\le j}$ are defined as:
$$M(g)_{ij}=\frac{j!}{i!}B_{i,j}\big(g_1, 2!g_2,\ldots,(i-j+1)!g_{i-j+1}\big).$$
\section{Hopf algebra interpretation}
\subsection{Hopf algebras}
\label{ssec:Hopf}
We give here a brief account of the subject. A standard reference on Hopf algebras is Sweedler's monograph \cite{Sweedler}.
An \textsl{associative unital algebra}, which we will abbreviate as \textsl{algebra} (in the symmetric monoidal category 
of vector spaces over a field $\K$, the morphisms being linear maps), is a 
$\K$-vector space $A$ together with:
\medskip
\begin{itemize}
\item 
a multiplication 
$m:A\otimes A\longrightarrow A, (a,b)\mapsto m(a,b)=a\cdot b$ such that 
$m (m\otimes \id)= m (\id\otimes m)$, i.e. 
$$(a\cdot b)\cdot c = a\cdot (b\cdot c) \qquad\mbox{for any $a,b\in A$};$$
\item
a unit $i:\K\hookrightarrow A$ such that $i(ts)=i(t)\cdot i(s)$ 
for any $s,t\in \K$ and such that, if we set $1_A=i(1_{\K})$, we have 
$a\cdot 1_A = a = 1_A\cdot a$ for any $a\in A$.
\end{itemize}
\bigskip  

\noindent
A (coassociative counital) \textsl{coalgebra} is a $\K$-vector space $C$ 
together with:
\medskip
\begin{itemize}
\item 
a comultiplication (or coproduct)
$\D:C\longrightarrow C\otimes C, c\mapsto \D(c)=\sum c_{(1)}\otimes c_{(2)}$ 
such that $(\D\otimes\id)\D =(\id\otimes \D)\D$.
\item
a counit $\varepsilon:C\longrightarrow \K$ such that $(\varepsilon\otimes I)\Delta=(I\otimes \varepsilon)\Delta=\mop{Id}_C$.
\end{itemize}
\medskip

For example, if $C$ is any coalgebra with 
comultiplication $\D$ and counit $\varepsilon$, its linear dual $A=C^*$ is an
algebra with product $m=\D^*$ and unit $i=\varepsilon^*$. Similarly, if $A$ is a \textsl{finite-dimensional}  algebra with product $m$ 
and unit $i$, then $C=A^*=\Hom_{\K}(A,\K)$ is a coalgebra with 
comultiplication $\D=m^*: A^*\longrightarrow A^*\otimes A^* \cong (A\otimes A)^*$ 
given by $\D(f)(a\otimes b) = f(a\cdot b)$, and counit 
$\varepsilon=i^*:A^*\longrightarrow \K$ given by 
$\varepsilon (f) = f(1_A)$.
\medskip

\noindent
A Hopf algebra is a $\K$-vector space $H$ which is both an algebra 
and a coalgebra, such that:
\begin{itemize}
\item 
The comultiplication $\D$ and the counit $\varepsilon$ are algebra morphisms, 
i.e:
\begin{align*}
& \D(a\cdot b)=\D(a)\cdot \D(b), \qquad \D(1)=1\otimes 1 \\ 
& \varepsilon(a\cdot b)=\varepsilon(a)\cdot \varepsilon(b),\qquad 
\varepsilon(1)=1, 
\end{align*}
where the product on $H\otimes H$ is given by 
$(a_1\otimes a_2)\cdot (b_1\otimes b_2)=(a_1\cdot b_1) \otimes (a_2\cdot b_2)$
for any $a_1,a_2,b_1,b_2\in H$.
\item
There is a linear map $S:H\longrightarrow H$, called the \textsl{antipode}, 
such that $m(S\otimes\id)\D=i\varepsilon=m(\id\otimes S)\D$.
\item 
The antipode is both an antimorphism of algebras and an antimorphism 
of coalgebras, i.e., for any $a,b\in H$, we have 
\begin{align*}
& S(a\cdot b)=S(b)\cdot S(a), \qquad S(1_H)=1_H, \\ 
& \D(S(a)) = (S\otimes S)\D^{op}(a), \qquad \varepsilon(S(a))= \varepsilon(a), 
\end{align*}
where $\D^{op}=\tau\circ \D$, with $\tau(u\otimes v)=v\otimes u$.
\end{itemize}
The dual $A^*$ of any finite-dimensional Hopf algebra 
$(A,m,i,\D,\varepsilon,S)$ is also a Hopf algebra, with mutiplication $\D^*$, 
comultiplication $m^*$, unit $\varepsilon^*$, counit $i^*$ and antipode $S^*$. 
Let us give some classical examples of Hopf algebras:
\begin{enumerate}
\item 
The universal enveloping algebra $U(\gLie)$ of a Lie algebra $\gLie$. 
By definition, $U(\gLie)=T(\gLie)/J$ is the quotient of the free associative 
algebra on the vector space $\gLie$ by the two-sided ideal $J$ generated by 
$\{x\otimes y-y\otimes x-[x,y],\ x,y\in\gLie\}$. 
The comultiplication, the counit and the antipode are the algebra morphisms 
(antimorphism for the antipode) given on the generators $x\in\gLie$ by 
\begin{align*}
\D(x) = x\otimes 1+1\otimes x, \qquad \varepsilon(x)=0, \qquad 
S(x)=-x. 
\end{align*}
The elements of $\gLie$ are said to be \textsl{primitive} in the Hopf algebra, 
because of the special form the comultiplication$\D$ takes on them.  
This Hopf algebra is \textsl{cocommutative}, i.e.~$\D(u)=\D^{op}(u)$ 
for any $u\in U(\gLie)$. 
\item
The group algebra $\K G$ of a finite group $G$. By definition, 
$\K G$ is the vector space freely generated by the elements of $G$, with multiplication bilinearly induced from the group law.
The comultiplication, the counit and the antipode are the algebra morphisms 
(antimorphism for the antipode) given on the generators $x\in G$ by 
\begin{align*}
\D(x) = x\otimes x, \qquad \varepsilon(x)=\delta_{x,1_G}, \qquad 
S(x)=x^{-1}. 
\end{align*}
The elements of $G$ are said to be \textsl{group-like} in the Hopf algebra, 
because of the special form the $\D$ takes on them.  
This Hopf algebra is also cocommutative.
\item
The algebra of representative functions $R[G]$ of a group $G$. 
By definition, $R[G]$ is the algebra generated by the matrix elements of 
finite-dimensional representations of $G$. The remarkable property 
$R[G]\otimes R[G] \cong \R[G\times G]$ (see \cite[Lemma 3.1.1]{Cartier}) allows us to define the coproduct 
of a function (of one variable) as being a function of two variables. 
Explicitly, the coproduct, the counit and the antipode of any representative 
function $f$ on $G$ are defined as follows, for any $x,y\in G$: 
\begin{align*}
\D(f)(x,y) = f(xy), \qquad \varepsilon(f)=f(1_G), \qquad 
Sf(x)=f(x^{-1}). 
\end{align*}
This Hopf algebra is commutative and is also called the 
\textsl{coordinate Hopf algebra} of the group $G$. 

If $G$ is a finite group, any function on $G$ is representative, and the 
Hopf algebra $R[G]$ is dual to the Hopf algebra $\K G$. If $G$ is a compact Lie group over $k=\mathbb R$, "almost all" functions on $G$ are representative, in the sense that there is a nondegenerate pairing $\langle\cdot,\cdot\rangle$ between $R[G]$ and the enveloping algebra $U(\frak g)$, where $\frak g$ is the Lie algebra of $G$, such that $\langle \delta(f),u\otimes v\rangle=\langle f,uv\rangle$ and $\langle f\otimes g,\Delta u\rangle=\langle fg,u\rangle$ for any $f,g\in R[G]$ and any $u,v\in U(\frak g)$. The pairing is given by $\langle f,u\rangle:=u.f(e)$, where $u$ acts on $f$ as a (left-invariant) differential operator, and where $e$ is the unit of the group \cite[3.3 and 3.4]{Cartier}.
\end{enumerate}

\subsection{Affine group schemes and proalgebraic groups}

Let $H$ be a commutative Hopf algebra and $A$ a commutative unital algebra 
over some field $\K$. Then, the vector space ${\cal{L}}(H,A)$ of $\K$-linear 
maps from $H$ to $A$ inherits a canonical unital algebra structure. The multiplication is 
given by the \textsl{convolution product}, defined, for any two 
$\alpha,\beta\in {\cal{L}}(H,A)$ as 
$$\alpha\ast\beta:=m_A\circ (\alpha\otimes\beta)\circ \D_H.$$ 
The convolution unit is the linear map $e_A:=i_A\circ \varepsilon_H$. 
The vector space ${\cal{L}}(H,A)$ is often huge 
and not easy to handle. However it contains a subgroup which is, roughly 
speaking, equivalent to $H$. In fact, let $\CAlg$ be the category of commutative and unital associative 
$\K$-algebras and consider the subset 
$$G_H(A):=\mop{Hom}_{\smop{CAlg}}(H,A)$$ 
of \textsl{$A$-valued characters of $H$}, i.e.~unital algebra morphisms 
from $H$ to $A$. It can easily be checked that the convolution of two 
characters is still a character and that the convolution inverse of a 
character is still a character, because $H$ and $A$ are commutative algebras. 
The unit $e_A=i_A\circ \varepsilon_H$ of ${\cal{L}}(H,A)$ is also 
a character, and the inverse is given by $\varphi\mapsto\varphi\circ S$ where $S$ is the antipode. This construction is functorial, that is, for any unital commutative 
algebras $A,B$ and for any commutative unital algebra morphism $f:A\to B$, 
there is a group morphism $G_H(f):G_H(A)\to G_H(B)$.\\

An \textsl{affine group scheme} \cite{Waterhouse} is a functor from the category of commutative algebras to the 
category of groups, representable by a commutative Hopf algebra $H$.\\

The group $G_H(A)$ has the particular property that $H$ can then be regarded 
as an algebra of $A$-valued functions on the group, because 
any element $h\in H$ can be seen as a function acting on $\alpha\in G_H(A)$ 
as $h(\alpha):=\alpha(h)$, and the fact that $\alpha$ is an algebra morphism, 
that is $\alpha(h\cdot h')=\alpha(h)\cdot\alpha(h')$ for any $h,h'\in H$, 
guarantees that we recover the usual definition of the product of functions, 
that is, $(h\cdot h')(\alpha) = h(\alpha)\cdot h'(\alpha)$. Moreover, the 
coproduct on $H$ then coincides with the coproduct on functions induced by 
the group law in $G_H(A)$, that is, $\D(h)(\alpha,\beta)=h(\alpha\ast\beta)$, 
because of the definition of the convolution product in $G_H(A)$. 
Therefore, any group $G_H(A)$ behaves like a compact Lie group, and $H$ like 
the algebra of its representative functions.\\

An \textsl{algebraic group} is by definition a functor $G$ from the category 
of commutative algebras to the category of groups, having the property that 
there exists a commutative and finitely generated Hopf algebra $H$ such that 
the group $G(A)$ is isomorphic to $G_H(A)$ for any commutative unital algebra 
$A$. In other words, an algebraic group is a representable functor from the 
category of commutative algebras to the category of groups, such that 
the Hopf algebra which represents it is finitely generated. In this case, 
the algebra $H$ is the algebra of $A$-valued polynomial functions on 
the group $G_H(A)$, and its generators can be seen as coordinate functions 
on the group. Classical examples of algebraic groups are the matrix groups 
$GL_n$, $SL_n$, $SO_n$ and $SU_n$.\\

A \textsl{pro-algebraic group} is a projective limit of algebraic groups. 
It can be viewed as a representable functor from the category of commutative 
algebras to the category of groups, such that the Hopf algebra which 
represents it is an inductive limit of finitely generated Hopf algebras. 
Hence the group has infinitely many coordinate functions. As any commutative Hopf algebra is an inductive limit of finitely generated hopf algebras, any affine group scheme is pro-algebraic \cite[Theorem 3.3]{Waterhouse}.\\

Finally, a pro-algebraic group $G_H$ is called \textsl{prounipotent} (see e.g. \cite[Paragraph 3.9]{Cartier}) if the 
Hopf algebra $H$ is \textsl{connected}, which means that there exists an 
increasing filtration of vector spaces 
$$H^0\subset\cdots\subset H^n\subset\cdots$$ 
with $H=\bigcup_{n\ge 0}H^n$, such that the vector space $H^0$ is 
one-dimensional, i.e.~$H^0=\K\cdot 1_H$, and such that the multiplication and 
the coproduct are of filtration degree zero, i.e.~for any $p,q\ge 0$ we have
\begin{equation*}
H^p\cdot H^q\subset H^{p+q}\qquad\mbox{and}\qquad 
\Delta(H^n)\subset\sum_{p+q=n}H^p\otimes H^q.
\end{equation*}
It is well known that the antipode of a connected Hopf algebra ``comes for 
free'' by recursive relations (in other words, any connected filtered 
bialgebra is a Hopf algebra), and is also of filtration degree zero, i.e.~$S(H^n)\subset H^n$ for any $n\ge 0$ \cite{FGV, Manchon08}.

\subsubsection{The Hopf algebra of symmetric functions}

A basic example of a pro-algebraic group is the group $\Ginv$ 
of invertible series in one variable with pointwise multiplication: for any 
commutative unital algebra $A$, the group $\Ginv(A)$ is the set of 
formal series 
$$f(t)=1_A+\sum_{n\ge 1}f_nt^n, \qquad f_n\in A,$$
endowed with the product
$$(fg)(t)=f(t)g(t)=1_A+\sum_{n\ge 1}\left(\sum_{k=0}^nf_kg_{n-k}\right)t^n,$$
with the convention $f_0=g_0=1_A$. Its representing Hopf algebra is 
the polynomial ring $\Hinv=\K[x_1,x_2,x_3\ldots]$ on infinitely many variables 
$x_n$, $n\geq 1$, which play the role of $A$-valued coordinate functions 
on each group $\Ginv(A)$, that is, 
$$x_n:\Ginv(A)\to A, f\mapsto x_n(f):=f_n.$$
The coproduct and the counit on $\Hinv$ are defined, for any 
$f,g\in \Ginv(A)$, as 
$$\Dinv(x_n)(f,g)=x_n(fg) \qquad\mbox{and}\qquad \varepsilon(x_n)=x_n(1),$$ 
where the constant series $1(t)=1_A$ is the unit in $\Ginv(A)$, 
and can be easily computed as 
$$\Dinv(x_n)=\sum_{p=0}^n x_p\otimes x_{n-p} \qquad\mbox{and}\qquad 
\varepsilon(x_n)=\delta_{n,0},$$ 
where we assume that $x_0=1$. 
The resulting bialgebra is graded, if we set the degree of a monomial 
$x_{n_1}^{d_1}\cdots x_{n_k}^{d_k}\in \Hinv$ to be $\sum_{i=1}^k n_i d_i$, that is, 
if we require the generators $x_n$ to be graded, with degree $|x_n|=n$. 
In degree zero there are only scalars, therefore $\Hinv$ is connected and the 
antipode is automatically defined by the recursive formula 
$$S(x_n)=-\sum_{p=0}^{n-1} S(x_p)x_{n-p}.$$ 
Hence $\Ginv$ is a pronilpotent group. The Hopf algebra $\Hinv$ is both commutative and cocommutative, and coincides 
with the well-known \textsl{Hopf algebra of symmetric functions}. 

\subsection{The Fa\`a di Bruno Hopf algebra}
\label{ssect:FdB}

Another key example of a pro-algebraic group is the group $\Gdif$ of 
formal diffeomorphisms tangent to the identity, and it is also prounipotent. 
Indeed, it is a functor: for any commutative unital algebra $A$, the group 
$\Gdif(A)$ is the set of formal series with coefficients in $A$, of the form 
$$f(t)=t + \sum_{n\geq 2} f_n\ t^n\in tA[[t]],$$ 
equipped with the usual composition law. It is the projective limit of the 
algebraic groups ${\Gdif}^{(n)}(A)$, for $n\geq 1$, which are the image of 
$\Gdif(A)$ in the quotient $tA[[t]]/t^{n+1}A[[t]]$. Furthermore, the functor 
$\Gdif$ is representable: its representing Hopf algebra $H$ is the Hopf algebra 
of $A$-valued polynomial functions on $\Gdif(A)$, freely generated as a commutatve algebra by the functions
$$x_n:\Gdif(A)\to A, f\mapsto x_n(f):= f_{n+1}, \qquad\mbox{for $n\geq 1$}.$$
The coproduct is induced by 
the composition of series: 
$$\D(x_n)(f\otimes g)= x_n(f\circ g)\quad\mbox{for any $f,g\in\Gdif(A)$},$$
counit is induced by the unit $\id(t)=t$ of the group: 
$$\varepsilon(x_n)=x_n(\id),$$ 
and antipode is induced by the composition-inversion of series: 
\begin{equation}\label{lagrange2}
S(x_n)=x_n(f^{-1})\quad\mbox{for any $f\in\Gdif(A)$}.
\end{equation}
The Lagrange inversion formula \eqref{Lagrange-formula} gives an explicit formula for $S(x_n)$. 
Using eq.~\eqref{eq:comp} we can explicitly compute the coproduct. Up to a 
flip between left and right terms in the tensor product $H\otimes H$ (caused by the fact tht we write composition of functions from right to left), this 
Hopf algebra coincides with the Fa\`a di Bruno Hopf algebra 
$$\HFdB=\K[x_1,x_2,x_3,...],$$
with coproduct and counit given, on the generators, by 
\begin{align}
\nonumber
\DFdB (x_n)&=\sum_{m=0}^n \frac{(m+1)!}{(n+1)!} 
\BB_{n+1,m+1}(1,2!x_1,3!x_2,...,(n-m+1)!x_{n-m})\otimes x_m \\ 
&=\ \sum_{m=0}^n 
\left(\sum_{{\scriptstyle k_0+k_1+\cdots +k_n=m+1, \atop \scriptstyle k_1+2k_2+\cdots +nk_n=n-m}} 
\dfrac{(m+1)!}{k_0! k_1!\cdots k_n!}\ x_1^{k_1}\cdots x_n^{k_n}\right) 
\otimes x_m,\label{doubilet} \\ 
\varepsilon(x_n) &= \delta_{n,0},\notag
\end{align}
where we use the convention $x_0=1$, and where $\BB$ denotes the Bell 
polynomial. The antipode can be found recursively, because $\HFdB$ is a graded 
and connected bialgebra, where the generators have degree $|x_n|=n$, 
for $n\ge 0$. The left-right flip of factors means that $\HFdB$ is in fact 
the representative Hopf algebra of the opposite pro-algebraic group 
$(\Gdif)^{op}$, with \textsl{opposite composition law} 
$f\circ^{op} g = g\circ f$. For any commutative algebra $A$, the group 
$\Gdif(A)$ can then be seen as the group of 
$A$-valued characters of $\HFdB$ with reversed convolution product: any formal series $f\in\Gdif(A)$ 
is identified with a linear map $f:\HFdB\to A$ satisfying the properties 
$$
f(1)=1_A \qquad\hbox{and}\qquad f(xy)=f(x)f(y)
\quad\hbox{for any $x,y\in\HFdB$}.
$$ 
The composition of two series $f,g\in\Gdif(A)$, then, coincides with their 
opposite convolution product, namely
$$f\circ g=m_A\circ (g\otimes f)\circ \DFdB.$$ 

Detailed explanations can also be found in \cite[Section II]{FGV2}. The Fa\`a di Bruno Hopf algebra first appeared under this name in 1974 
(P.~Doubilet, \cite{Doubilet}), and in 1979 (S.A.~Joni and G.-C.~Rota, 
\cite{JoniRota}), as an example of a Hopf algebra structure obtained from the lattice 
of set partitions. We say more about such Hopf algebras in section 
\ref{Section6}. It is also worth mentioning that the Fa\`a di Bruno Hopf algebra is a Hopf subalgebra of the Connes-Moscovici Hopf algebra, which is neither commutative nor cocommutative (see \cite[Paragraph III]{CM} for a detailed account, see also \cite{Menous}). The Fa\`a di Bruno Hopf algebra also plays a role in applied mathematics, e.g. in control theory \cite{GD}.

\subsection{Two more examples}
\subsubsection{The Hopf algebra of rooted trees}\cite{ConnesKreimer,F,Manchon08}
\label{sssect:bck}
A \textsl{rooted tree} is a class of oriented (non planar) graphs with a finite number 
of vertices, among which one is distinguished and called the \textsl{root}, 
such that any vertex admits exactly one incoming edge, except the root which 
has no incoming edges. Any tree yields a poset structure on the set of its vertices:  two vertices $x$ and $y$ verify $x \le y$ if and only if there is a path from a root to $y$ passing through $x$. Two graphs are equivalent (hence define the same rooted tree) if and only if the two underlying posets are isomorphic. Here is the list of rooted trees up to five vertices:
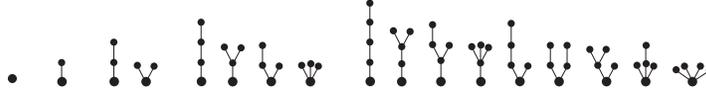
\begin{equation*}
\racine \hskip 5mm \arbrea \hskip 5mm  \arbreba\  \arbrebb \hskip 5mm  
\arbreca\  \arbrecb\  \arbrecc\ \arbrecd \hskip 5mm  \arbreda\  \arbredb\  
\arbredc\  \arbredd\  \arbrede\  \arbredf\  \arbredz\  \arbredg\ \arbredh
\end{equation*}
A \textsl{rooted forest} is a finite collection of rooted trees. 
The empty set is the forest with containing no trees, and is denoted by $\un$. 
The \textsl{grafting operator} $B_+$ takes any forest and returns the
tree obtained by grafting all components onto a common root. In particular,
$B_+(\un)=\racine$. 
\medskip 

Let $\Cal T$ denote the set nonempty rooted trees and let $\HRT=\K[\Cal T]$ 
be the free commutative and unital algebra generated by the elements of $\Cal T$. 
We identify a product of trees with the forest consisting of these trees. 
Therefore the vector space underlying $\HRT$ is the linear span of rooted 
forests. This algebra is a graded and connected Hopf algebra, called the 
\textsl{Hopf algebra of rooted trees}, with the following structure. 
The grading is given by the number of vertices of trees. 
The coproduct on a rooted forest $u$ (i.e. a product of rooted trees) is 
described as follows: the set $\Cal V(u)$ of vertices of a forest $u$ is 
endowed with the partial order defined by $x \le y$ if and only if there is a 
path from a root to $y$ passing through $x$. Any subset $W$ of the set of 
vertices $\Cal V(u)$ of $u$ defines a {\sl subforest\/} $u\restr{W}$ of $u$ 
in an obvious manner, i.e. by keeping the edges of $u$ which link two elements 
of $W$. The poset structure is given by restriction of the partial order to $W$, and the minimal elements are the roots of the subforest. The coproduct is then defined by:
\begin{equation}
\label{coprod}
\DRT(u)= \sum_{V \amalg W=\Cal V(u) \atop W<V} u\restr{V}\otimes u\restr{W}.
\end{equation}
Here the notation $W<V$ means that $y<x$ for any vertex $x$ in $V$ and any
vertex $y$ in $W$ such that $x$ and $y$ are comparable. Note that both $\emptyset < V$ and $V<\emptyset$. Such a couple $(V,W)$
is also called an \textsl{admissible cut}, with \textsl{crown} (or pruning) 
$u\restr{V}$ and \textsl{trunk} $u\restr{W}$. We have for example:
 \allowdisplaybreaks{
\begin{eqnarray*}
\DRT\big(\arbrea\big) &=& 
   \arbrea \otimes \un + \un \otimes \arbrea + \racine \otimes \racine \\
\DRT\big(\! \arbrebb \big) &=& \arbrebb \otimes \un + \un \otimes \arbrebb + 
   2\racine \otimes\arbrea + \racine\racine\otimes \racine . 
\end{eqnarray*}}
The counit is $\varepsilon(\un)=1$ and $\varepsilon(u)=0$ for any non-empty 
forest $u$. The coassociativity of the coproduct is easily checked using an iterated 
formula for the restricted coproduct 
$$
\wtDRT(u) = \DRT(u)-u\otimes\un -\un\otimes u
= \sum_{V\amalg W=\Cal V(u) \atop W<V,\, V,W\not =\emptyset} 
u\restr{V}\otimes u\restr{W}, 
$$ 
where the restriction that $V$ and $W$ are nonempty means that $V$ and $W$ 
give rise to an ordered partition of $\Cal V(u)$ into two blocks. 
In fact, the iterated restricted coproduct writes in terms of ordered 
partitions of $\Cal V(u)$ into $n$ blocks:
$$ 
\wtDRT^{n-1}(u)=
\sum_{V_1\amalg\cdots\amalg V_n=\Cal V(u) \atop 
V_n<\cdots <V_1,\,V_j\not =\emptyset} 
u\restr{V_1}\otimes\cdots\otimes u\restr{V_n},
$$ 
and we get the full iterated coproduct $\DRT^{n-1}(u)$ by allowing empty blocks 
in the formula above. Note that the relation $<$ on subsets of vertices is not 
transitive. The notation $V_n<\cdots <V_1$ is to be understood as $V_i<V_j$ for 
any $i>j$, with $i,j\in\{1,\ldots ,n\}$. 
\medskip 

This Hopf algebra first appeared in the work of A. D\" ur in 1986 \cite{Dur}, 
as an incidence Hopf algebra (see section 6). 
It has been rediscovered and intensively studied by D.~Kreimer in 1998 
\cite{Kreimer}, as the Hopf algebra describing the combinatorial part 
of the BPHZ renormalization procedure of Feynman graphs in a scalar 
$\varphi^3$ quantum field theory. 
D.~Kreimer and A. Connes also proved in \cite{ConnesKreimer} that the 
operator $B_+$ satisfies the property
\begin{equation}
\label{cocycle}
\DRT\big(B_+(t_1\cdots t_n)\big) = 
B_+(t_1\cdots t_n)\otimes\un + (\mop{Id}\otimes B_+)\circ\DRT(t_1\cdots t_n), 
\end{equation}
for any $t_1,...,t_n\in \Cal T$. This means that $B_+$ is a 1-cocycle in the 
Hochschild cohomology of $\HRT$ with values in $\HRT$, and the couple 
$(\HRT,B_+)$ is then proved to be universal among commutative Hopf algebras 
endowed with a 1-cocycle. 
\bigskip 

The corresponding proalgebraic group $\GRT=\Hom_\CAlg(\HRT,-)$ first appeared 
in numerical analysis in the work of J. Butcher in 1963 \cite{Butcher63, Butcher, Brouder}, where it was related to the change of Runge-Kutta methods for 
solving ordinary differential equations. 
It was then studied by F. Chapoton, who related this group to the pre-Lie 
operad \cite{Chapoton,ChapotonLivernet}, see section \ref{sect:operad}.\\

For any commutative and unital algebra $A$, the group $\GRT(A)$ can be identified with the set of \textsl{formal series expanded over rooted trees with 
coefficients in $A$}, i.e. the set of maps form the set of nonempty rooted trees to $A$. These \textsl{B-series} are by now widely used in the study of approximate solutions of nonlinear differential equations\footnote{A B-series is map  from the set of rooted trees, including the empty one, to the base field $\mathbb R$ or $\mathbb C$ of real or complex numbers. A B-series can be identified with an element of the Butcher group if and only if the coefficient of the empty tree is equal to one.}\cite{HLW}. Such a map extends multiplicatively in a unique way to an $A$-valued character of $\HRT$. 
There is an injective 
morphism of graded Hopf algebras 
\begin{equation}
\label{fdb-bck}
\Psi:\HFdB\inj 6 \HRT,  
\end{equation}
as proven in Paragraph \ref{lie-rt} below. This in turn induces a surjective group homomorphism
$$\GRT(A)\to\!\!\!\!\!\to \Gdif(A).$$  
for any commutative algebra $A$.
\subsubsection{Feynman graphs and diffeographisms \cite{ConnesKreimerII}}
A \textsl{Feynman graph} is a (non planar) graph with a finite number of 
vertices, together with internal edges (with both ends attached 
to a vertex) and external edges (with one loose end, the other one being attached to a vertex). 
A graph is \textsl{connected} if any vertex can be reached from any other 
vertex by following internal edges. 
A connected graph is called \textsl{one-particle irreducible} (1PI) if it 
remains connected after removing any internal edge \cite{ConnesKreimerI,
Manchon08}.\\

In quantum field theory, Feynman graphs are a useful tool to describe some 
integrals, the \textsl{amplitudes} of the graphs, which compute the 
expectation value of the events schematically represented by the graphs. 
The \textsl{Feynman rules}, which recover the amplitude from a given graph, 
associate to loose ends (viewed as external vertices) some fixed points in 
the space-time manifold, to (internal) vertices some variable points to be 
integrated over the space-time manifold, and to edges the propagators of the 
field theory considered. After a Fourier transform, external edges are 
associated to fixed external momenta, and internal edges to variable momenta, 
which are to be integrated. 
A field theory is prescribed through a Lagrangian, which determines the types 
of Feynman graphs to be considered: the valence of the internal edges and the 
possible types of edges. The physically most interesting field theories give 
rise to Feynman graphs whose amplitude is a divergent integral, and which are 
therefore called \textsl{divergent graphs}. Physicists have developed some 
techniques, called \textsl{renormalization}, to extract a meaningful finite 
contribution from such graphs. 
A common step of any renormalization process consists in a purely combinatorial manipulation  
of graphs, in order to take care of possible divergent subgraphs. 
D. Kreimer proved in \cite{Kreimer} that this step can be efficiently 
described in terms of the Hopf algebra of rooted trees presented in section 
\ref{sssect:bck}, where rooted trees encode the hierarchy of divergent 
subgraphs. A. Connes and D. Kreimer then proved in \cite{ConnesKreimer,
ConnesKreimerI,ConnesKreimerII} that the renormalization Hopf algebra can be 
defined directly on Feynman graphs.\\

The so-called \textsl{Connes-Kreimer Hopf algebra}, or \textsl{Hopf algebra 
of Feynman graphs}, is the free commutative and unital algebra 
$\HFG = \K[\Cal F]$ generated by the set $\Cal F$ of connected and 1PI graphs 
which are \textsl{superficially divergent}, i.e. graphs which are divergent 
even after having renormalized any divergent subgraph. For any given field 
theory, the superficial divergence of a Feynman graph can be stated as a 
combinatorial constraint, and implies, for instance, that the number of external 
legs does not exceed some critical number. 
The product of 1PI divergent graphs is given by their disjoint union. Hence $\HFG$ is the vector space generated by superficially divergent locally 1PI graphs, connected or not, including the empty graph.
The coproduct on $\HFG$ is given on any generator $\Gamma\in \Cal F$ by:
\begin{equation*}
\Delta(\Gamma)=\un\otimes\Gamma+\Gamma\otimes\un
   +\sum_{\emptyset\subsetneq\gamma\subsetneq\Gamma}\Gamma/\gamma,
\end{equation*}
where the sum runs over the set of proper and non-empty subgraphs of $\Gamma$ 
having connected components which are 1PI and superficial divergent. The 
\textsl{contracted graph} $\Gamma/\gamma$ is obtained by shrinking each
connected component of $\gamma$ onto a point
\cite{ConnesKreimerI, Manchon08}. 
Note that the condition which ensures that a contracted graph $\Gamma/\gamma$ 
belongs to $\HFG$ coincides with the requirement that $\gamma$ only has  
superficially divergent connected components \cite{ConnesKreimerI, Agarwala}. \\

The counit on $\HFG$ is given on any generator $\Gamma\in \Cal F$ by 
$\varepsilon(\Gamma)=0$, and the bialgebra thus obtained is graded by the 
\textsl{loop number} of the graphs\footnote{The loop number is defined by $L(\Gamma)=I(\Gamma)-V(\Gamma)+|\pi_0(\Gamma)|$, where $I(\Gamma$) stands for the number of internal edges, $V(\Gamma)$ stands for the number of vertices, and $|\pi_0(\Gamma)|$ denotes the number of connected components of the graph.}, and connected. Therefore it is automatically 
a (commutative) Hopf algebra, with antipode defined recursively. 
The convolution group $\GFG = \Hom_\CAlg(\HFG,-)$ 
is another example of a prounipotent group, called 
the group of \textsl{diffeographisms} by Connes and Kreimer in 
\cite{ConnesKreimerII}.\\

This affine group scheme plays a key role in the description of the renormalization of quantum fields. Indeed, any splitting $A=A_-\oplus A_+$ of the target algebra into two subalgebras (with $1\in A_+$) yields a Birkhoff decomposition:
$$\GFG={\GFG}_-\cdot{\GFG}_+,$$
with ${\GFG}_\pm=\Hom_\CAlg(\HFG,A_\pm)$. For any $\varphi=\varphi_-^{-1}\varphi_+\in\GFG$, the renormalized character is the component $\varphi_+$, whereas the component $\varphi_-$ is the counterterm character \cite{ConnesKreimerI}. The combinatorics underlying this renormalization process relies on a Rota-Baxter property for the projection $A\to\!\!\!\!\!\to A_-$, and applies to any connected graded Hopf algebra in full generality (see \cite{Ebrahimi-FardGuo, Manchon08} and see \cite{EP1, EP2} for a generalization outside the Rota-Baxter framework).\\

The Hopf algebra of Feynman graphs is naturally endowed with several 
operators analogous to the grafting operator $B_+$ defined on the Hopf algebra 
of rooted trees. In fact, for any primitive graph $\gamma$, i.e.~such that 
$\DFG(\gamma)=\gamma\otimes\un+\un\otimes\gamma$, one can define an 
\textsl{insertion operator} $\delta\mapsto B_\gamma(\delta)$ by summing up all 
possibilities to insert $\gamma$ inside $\delta$, and taking into account some 
appropriate symmetry factors (see L. Foissy's notes in the present volume). 
The insertions are so chosen that, for the resulting graph $\Gamma$, the 
subgraph $\delta$ appears as the contracted graph $\Gamma/\gamma$.
For scalar or fermionic field theories, each operator $B_\gamma$ 
satisfies the same cocycle property \eqref{cocycle} as $B_+$, namely:
\begin{equation}\label{cocycle2}
\Delta\big(B_\gamma(\delta)\big) = 
   B_\gamma(\delta)\otimes\un+(\mop{Id}\otimes B_\gamma)\circ\Delta(\delta).
\end{equation}
For gauge theories, equation \eqref{cocycle2} takes place in the quotient 
$\ovHFG$ of $\HFG$ by the ideal describing the Ward-Takahashi identities 
of the gauge theory \cite{VanSuijlekom}. Due to the universal property 
of the Hopf algebra $\HRT$ with the operator $B_+$ \cite{F}, for any primitive 
graph $\gamma\in\HFG$ there is a unique graded Hopf algebra morphism
\begin{equation}
\Phi_\gamma:\HRT\longrightarrow \HFG \qquad\mbox{\big(resp. to $\ovHFG$\big)}
\end{equation}
such that $\Phi_\gamma\circ B_+=B_\gamma\circ \Phi_\gamma$. Composing $\Phi_\gamma$ 
with the embedding $\Psi:\HFdB\inj 6\HRT$, we get an injective 
Hopf algebra morphism
\begin{equation*}
\Phi_\gamma\circ\Psi:\HFdB\inj 6 \HFG 
\qquad\mbox{\big(resp. to $\ovHFG$\big)}.
\end{equation*}
This embedding was first found by Connes and Kreimer in 
\cite{ConnesKreimerII}. The reader will again find a more detailed exposition 
in L. Foissy's note in the present volume. See also \cite{BS} for a concrete 
application of these Hopf algebra morphisms in quantum field theory. Finally, the map $\Phi_\gamma\circ\Psi$ yields a surjective morphism from the 
group of diffeographisms (resp. a subgroup) to the group of formal 
diffeomorphisms: for any commutative and unital algebra $A$, there is a 
canonical projection of groups
\begin{equation*}
\GFG(A)\to\!\!\!\!\!\to \Gdif(A) 
\qquad\mbox{\big(resp. $\ovGFG(A)\to\!\!\!\!\!\to \Gdif(A)$\big)}.
\end{equation*}

\subsection{The non-commutative Fa\`a di Bruno Hopf algebra}
\label{ssect:HFdBnc}

Many of the free commutative Hopf algebras related to the Fa\`a di Bruno 
Hopf algebra admit a natural non-commutative analogue, which in general 
turns out to be simpler from a combinatorial and from an operadic point 
of view. \\

In \cite{F}, L. Foissy introduced a non-commutative version of the Hopf 
algebra $\HRT$ based on \textsl{planar rooted trees}, and therefore denoted 
by $\HPRT$, which has the remarkable property that it is self-dual. The set of generators of $\HPRT$ as a free noncommutative algebra (planar rooted trees) canonically projects itself onto the set of generators of $\HRT$ as a free commutative algebra (non-planar rooted trees).\\

In the case of the Fa\`a di Bruno Hopf algebra, the generators are labeled by  
natural numbers, and it is also the case for the non-commutative lift described in \cite{BFK}. 
This \textsl{non-commutative Fa\`a di Bruno Hopf algebra} is defined as 
the free associative algebra on the same generators as $\HFdB$, that is 
$\HFdBnc=\K\langle x_1,x_2,...\rangle$, with the usual counit 
$\varepsilon(x_n)=\delta_{n,0}$, where $x_0=1$, and the coproduct lifted as 
\begin{equation*}
\DFdBnc(x_n) = \sum_{m=0}^n
\left(\sum_{k_0+\cdots +k_m=n-m\atop k_i\geq 0}\!\! 
x_{k_0}\cdots x_{k_m}\right)\otimes x_m.
\end{equation*}
It is still a graded and connected bialgebra, and the recursive relation 
for the antipode yields the following \textsl{non-commutative Lagrange 
formula} (see \eqref{lagrange2} and remark therein), quite involved compared to its commutative counterpart \eqref{Lagrange-formula}:
$$
S(x_n) = -x_n-\sum_{k=1}^{n-1} (-1)^k \!\!\!\!\!\!\!\!
\sum_{n_1+\cdots +n_{k+1}=n \atop n_i> 0}\!\!\!\!\!\! 
\lambda(n_1,...,n_k)\ \ x_{n_1}\cdots x_{n_{k+1}},
$$
where 
$$
\lambda(n_1,...,n_k)=\sum_{m_1+\cdots +m_k=k\atop m_1+\cdots +m_h\geq h\hbox{ \sevenrm for }h<k} 
\binom{n_1+1}{m_1}\cdots \binom{n_k+1}{m_k}.
$$
It is then easy to check that $\HFdB$ coincides with the abelian quotient 
$\HFdBnc/[\HFdBnc,\HFdBnc]$, and therefore that the surjective algebra 
morphism $\pi: \HFdBnc\to\!\!\!\!\!\to\HFdB$ which identify the generators 
is a Hopf algebra morphism.\\

The commutative Hopf algebra $\Hinv$ admits a similar non-commutative lift 
$\Hinvnc=\K\langle x_1,x_2,...\rangle$, with exactly the same coproduct 
formula $\Dinv(x_n)=\sum_{p=0}^n x_p\otimes x_{n-p}$. The resulting 
non-commutative but still co-commutative Hopf algebra is well known 
under the name of \textsl{Hopf algebra of non-commutative symmetric functions}, the graded dual of which is the Hopf algebra of \textsl{quasi-symmetric functions} \cite{Gelfand}.\\

\subsection{Proalgebraic groups on non-commutative algebras}
Proalgebraic groups on non-commutative unital algebras in fact do exist, 
and should be exhibited by means of \textsl{free products of algebras}. This well-established fact \cite{Berstein, V87, Zhang} deserves some explanation, as it may be not so well-known to non-specialists.  Although the functor $\Ginv$ can be naturally extended to noncommutative algebras along these lines, there seem however to be no means to extend the functor $\Gdif$ to non-commutative algebras in this framework:  we come back to this point in Paragraph \ref{ssect:HFdBnc-question} below.\\

\subsubsection{Cogroups \cite{Berstein}}

The free product of associative algebras can be defined as the unique 
operation $\star$ between associative algebras such that for any two 
vector spaces $V$ and $W$, there is a canonical algebra isomorphism
$$T(V)\star T(W)\cong T(V\oplus W)$$
between associated tensor algebras. For any unital $\K$-algebras $A$ and $B$, 
the free product $A\star B$ can explicitely be defined as the algebra 
\begin{equation*}
A\star B := T(\overline A\oplus \overline B)/J,
\end{equation*}
where $\overline A=A/\K\cdot 1_A$ and $\overline B=B/\K\cdot 1_B$ are 
the augmentation ideals, and where $J$ is the ideal of the tensor algebra 
$T(\overline A\oplus \overline B)$ generated by 
$\{a\otimes a'-aa',\,a,a'\in \overline A\}$ and 
$\{b\otimes b'-bb',\,b,b'\in \overline B\}$. Any element of $A\star B$ is 
represented by an element of $T(\overline A\oplus\overline B)$, which can be written
as a finite sum of terms of the type:
\begin{equation}
\label{ecriture}
x_1\otimes\cdots\otimes x_r,
\end{equation}
where $x_j$ is in $\overline A$ if and only if $x_{j+1}$ is in $\overline B$ 
for $j=1,\ldots,r-1$, and vice-versa. We write such a term $x_1\cdots x_r$. 
The product of two elements in $A\star B$ is given by the concatenation of 
their representative elements in $T(\overline A\oplus \overline B)$. 
In particular, there are canonical algebra embeddings 
$A\cong A\star \K \subset A\star B$ and $B\cong \K\star B\subset A\star B$, 
and therefore also an embedding of vector spaces
$$j: A\otimes B \inj 6 A\star B,$$
where $j(A\otimes B)$ contains exactely the terms $x_1 \otimes x_2\in A\star B$
with $x_1\in A$ and $x_2\in B$. 
Note however that the map $j$ is not an algebra morphism, because the 
multiplication in $A\otimes B$ is $(m_A\otimes m_B)\circ \tau_{23}$, while 
the multiplication in $A\star B$ is the concatenation.\\

The free product is characterized by the following universal property: 
for any associative unital algebras $A,B,C$ and for any unital algebra 
morphisms $f:A\to C$ and $g:B\to C$, there is a unique algebra morphism 
$f\ \overline\star\ g:A\star B\to C$ which extends both $f$ and $g$. 
The algebra morphism $f\ \overline\star\ g$ is defined by
\begin{equation*}
f\ \overline\star\ g([x_1\otimes\cdots\otimes x_r]):=h(x_1)\cdots h(x_r),
\end{equation*}
where $h:\overline A\oplus\overline B\to C$ is defined as $h(x)=f(x)$ 
if $x\in\overline A$, and as $h(x)=g(x)$ if $x\in\overline B$. 
This implies that there exists also a canonical projection of vector spaces 
$$\pi:A\star B\to\!\!\!\!\!\to A\otimes B,$$
which consists in ``putting the elements of $\overline A$ on the left and 
the elements of $\overline B$ on the right'', namely:
\begin{eqnarray*}
\pi([x_1\otimes\cdots\otimes x_r])&=&
\left\{\begin{array}{ll}
x_1 x_3\cdots \otimes x_2 x_4\cdots & \hbox{if $x_1\in \overline A$},\\
x_2 x_4\cdots \otimes x_1 x_3\cdots & \hbox{if $x_1\in \overline B$}.
\end{array}\right. 
\end{eqnarray*}
It is easily checked that, contrarily to the embedding $j$, the projection 
$\pi$ is an algebra morphism.\\

Finally, the free product $\star$ is a bifunctor: given two unital algebra 
morphisms $f:A\to A'$ and $g:B\to B'$, there exists a straightforward 
algebra morphism $f\star g:A\star B\to A'\star B'$. The multiplication 
$m_A:A\otimes A\to A$ uniquely extends to an algebra morphism 
$\widetilde m_A:A\star A\to A$ such that $\widetilde m_A=m_A\circ \pi$. 
We also have $\widetilde m_A=\mop{Id}_A\overline\star \mop{Id}_A$, and for any 
unital algebra morphisms $f:A\to C$ and $g:B\to C$ we have:
$$
f\overline\star g=\widetilde m_C\circ (f\star g).
$$

A \textsl{cogroup in the category of unital associative $\K$-algebras}
\cite{EH, Berstein, BH}, also called \textsl{$H$-algebra} in \cite{Zhang}, 
is a unital associative $\K$-algebra $H$ together with a ``coproduct'' 
$\Delta:H\to H\star H$, a ``counit'' $\varepsilon:\K\to H$ and an 
``antipode'' $S:H\to H$ such that:
\begin{align*}
(\id\star\Delta)\circ\Delta &=(\Delta\star \id)\Delta:H\to H\star H\star H,\\
(\varepsilon\star \id)\circ \Delta &=(\id\star\varepsilon)\circ\Delta=\id,\\
\widetilde m_H\circ(S\star \id)\circ\Delta & 
=\widetilde m_H\circ(\id\star S)\circ\Delta=i\circ\varepsilon.
\end{align*}
A similar notion has been developed by D. Voiculescu under the name "dual group"  in the context of operator algebras \cite{V87}. The following proposition, which can be easily derived from the Yoneda lemma (see e.g. \cite[Paragraph III.2]{Maclane}), is well-known to experts. However we give a 
pedestrian proof adapted to our context:

\begin{prop}
Let $\mop{Sets}$ and $\Alg$ stand for the category of sets and unital associative algebras respectively. Let $H$ be a unital algebra and let $G$ be the functor 
\begin{align*}
\Alg & \longrightarrow  \mop{Sets} \\ 
A & \longmapsto  G(A) =  \Hom_{\Alg}(H,A).
\end{align*}
Then $G$ takes its values in the category of groups if and only if the representing object $H$ is a 
cogroup in the category of unital $\K$-algebras.
\end{prop}

\begin{proof}
Suppose that $H$ is a cogroup in the category of unital $\K$-algebras, 
and consider any associative unital algebra $A$. Then we set 
$g\ast h=\widetilde m_A\circ (g\star h)\circ \Delta$ for any unital algebra 
morphisms $g,h:H\to A$. The associativity of this product is obvious, 
as well as the fact that $1_A\circ\varepsilon$ is the unit. The fact that 
$g\ast h$ is still an algebra morphism comes from the fact that 
$\widetilde m_A$, $g\star h$ and $\Delta$ are. The inverse of any $g$ is 
given by $g\circ S$, which proves that $ \Hom_{\Alg}(H,A)$ is a group. 
\medskip 

Conversely, if the functor $G$ takes its values in the category of groupes and is representable as 
$\Hom_{\Alg}(H,-)$, consider the two embeddings $j_1,j_2:H\to H\star H$, 
which are algebra morphisms, obtained from the two obvious embeddings of 
$H$ into $T(\overline H\oplus\overline H)$. 
Then $j_1\star j_2:H\star H\to H\star H$ is manifestly equal to the identity. 
It is then easy to check that the ``coproduct'' $\Delta$ is given by the 
product $j_1\cdot j_2$ in the group $\Hom_{\Alg}(H,H\star H)$. 
The ``counit'' is the unit of the group $\Hom_{\Alg}(H,\K)$, and the 
``antipode'' is the inverse in the group $\Hom_{\Alg}(H,H)$.
\end{proof}

Note that any cogroup $H$ is also a Hopf algebra, by composing the 
``coproduct'' $\Delta$ on the left by the projection 
$\pi:H\star H\to\!\!\!\!\! \to H\otimes H$, thus getting a genuine coproduct 
on $H$. The converse is not true: if $H$ is a Hopf algebra with coproduct 
$\D$, the lift $\widetilde \D=j\circ\D: H\to H\star H$ is not necessary an 
algebra morphism, because $j:H\otimes H\inj{6} H\star H$ is not. 
Note also that any cogroup in the category of associative unital algebras 
is a free algebra \cite[Theorem 1.2]{Berstein} or \cite{Zhang}. 

\subsubsection{The cogroup of invertible series} 
One can note that the proalgebraic group $\Ginv$ 
still makes sense on associative unital algebras which are not commutative: 
for any algebra $A$ of this type, the pointwise multiplication of two 
formal series of the form $f(t)=1_A+\sum_{n\geq 1} f_n\ t^n$ with coefficients 
$f_n\in A$ is still an associative operation, and gives rise to a group 
$\Ginv(A)$ which is no longer abelian. One may then wonder about the existence of a functor which associates to 
any unital algebra $A$ (commutative or not) a convolution group 
$\Hom_{\Alg}(H,A)$ represented by a non-commutative Hopf algebra $H$, 
which would reproduce the group $\Ginv(A)$ for $H=\Hinvnc$, and which would 
give rise to a new group of formal diffeomorphisms for $H=\HFdBnc$. 
Such a functor cannot exist, because if $A$ is a non-commutative algebra, 
the convolution product of two elements in $\Hom_{\Alg}(H,A)$ has no reason 
to be an algebra morphism. In the next section we give an explanation of the existence of the non-abelian 
group $\Ginv(A)$ for $A$ non-commutative, but our argument does not justify 
the existence of the non-commutative Hopf algebra $\HFdBnc$ in terms of 
groups. We come back to this point in Paragraph \ref{ssect:HFdBnc-question} below.\\

For the non-commutative Hopf algebra $\Hinvnc=\K\langle x_1,x_2,...\rangle$,  
the free product $\Hinvnc\star \Hinvnc$ can be described as the free 
algebra on two types of variables, one for each copy of $\Hinvnc$, that is, 
$\Hinvnc\star \Hinvnc=\K\langle x_1,y_1,x_2,y_2...\rangle$. The coproduct 
$\Dinvstar: \Hinvnc\longrightarrow \Hinvnc\star \Hinvnc$ defined on 
generators as 
\begin{align*}
\Dinvstar(x_n)=\sum_{p=0}^n\ x_p\ y_{n-p}\ 
\in \Hinvnc\otimes \Hinvnc \subset \Hinvnc\star \Hinvnc, 
\end{align*}
is obviously coassociative, and makes $\Hinvnc$ into a cogroup. Its 
associated functor $\Ginv= \Hom_\Alg(\Hinvnc,-):\Alg\longrightarrow \Groups$ 
then defines the \textsl{proalgebraic group of invertible series with 
non-commutative coefficients}, namely, for any non-commutative unital 
associative algebra $A$, the non-abelian group 
$$\Ginv(A)=\Big\{ f(t)= \sum f_n t^n |\ f_0=1\Big\} \subset A[[t]]$$
with pointwise product $(fg)(t)=f(t)\ g(t)$. 
The Hopf algebra coproduct 
$\Dinv:\Hinvnc\longrightarrow \Hinvnc\otimes \Hinvnc$ 
is then just the composition of the cogroup coproduct $\Dinvstar$ 
by the canonical projection 
$\Hinvnc\star \Hinvnc \to\!\!\!\!\! \to \Hinvnc\otimes \Hinvnc$. 


\subsection{Open questions about the non-commutative Fa\`a di Bruno 
Hopf algebra}
\label{ssect:HFdBnc-question}

Contrarily to invertible series with ordinary multiplication, formal 
diffeomorphisms $\Gdif(A)$ with non-commuta\-tive coefficients do not form 
a group. Indeed, the composition is not associative! 
As a consequence, the noncommutative Fa\`a di Bruno coproduct 
$\DFdB^{\smop{nc}}$ cannot be derived from a cogroup coproduct 
$\Delta:\HFdBnc \longrightarrow \HFdBnc\ast \HFdBnc$ by composing it on the 
left with the projection $\pi$ onto $\HFdBnc\otimes \HFdBnc$. One may then ask the following questions: 
\begin{itemize}
\item 
Which properties do characterise the set $\Gdif(A)$ of formal series with 
coefficients in a non-commutative algebra $A$, generalizing those of a group, 
which would allow to define $\Gdif$ as a functor represented by $\HFdBnc$? 
\item
Is there a cogroup coproduct $\D:H\longrightarrow H\ast H$ such that 
$\pi\circ\Delta:H\to H\otimes H$ gives $\DFdB$, which would then allow to 
define a proalgebraic group of formal series 
$G(A)=\{ f(t)= t + \sum_{n\geq 2} f_n t^{n+1}\ |\ f_n\in A\}$ 
with coefficients in non-commutative algebras $A$, with 
group law induced by $\D$ and therefore \textsl{different} from the 
composition $(f\circ g)(t)=\sum f_n g(t)^{n+1}$? 
If such a group exists, it is probably necessary to rearrange the terms 
$f_n t^{n+1}$ of each series in some non-commutative fashion. 
\end{itemize}
These questions are still open. 

\section{Lie algebra interpretation}
\label{sect:witt}

\subsection{Vector fields, Witt and Virasoro Lie algebras}\label{witt-virasoro}
The reader will find a comprehensive treatment of these infinite-dimensional Lie algebras and their representation theory in \cite{KacRaina} or in \cite{GO}. The Lie algebra $\Vect(\S^1)$ of \textsl{regular vector fields on the circle} 
is, by definition, the Lie algebra of derivations of the ring of Laurent 
polynomials $\R[z,z^{-1}]$. It can be seen 
as the real vector space 
$$
\Vect(\S^1)=\Span_\R\Big\{ e_n= \frac{1}{2i\pi}e^{2i\pi  nt} \frac{\d}{\d t},\ n\in\Z \Big\}, 
$$
where $z=e^{2i\pi t}$. The Lie bracket is given by:
$$
[e_m,e_n]=(n-m)\ e_{n+m}.
$$
The same Lie algebra can be considered over any base field $\K$, and in the 
case of a field of characteristic $p>0$ it was first introduced by E. Witt 
in a Hamburg Seminar in 1939, and reported by H. Zassenhaus in 
\cite{Zassenhaus}, with generators restricted to $1\leq n \leq p-2$.\\

By extension, the complex Lie algebra $\Vect(\S^1)_\CC$, which can be 
understood as the Lie algebra of polynomial complex functions on the circle, 
is now called the \textsl{Witt Lie algebra} and denoted by $\Cal W(\CC)$. 
Its one-dimensional central extension is the so-called \textsl{Virasoro Lie 
algebra}, first found in positive characteristic by R.~E. Block in 1966 
\cite{Block}, and then in characteristic zero by I.M. Gelfand and D.B. Fuks 
in 1968 \cite{GelfandFuks}. These two Lie algebras describe the infinitesimal 
conformal transformations in two real dimensions (with or without anomaly), 
and play a central role in conformal field theory and in string theory \cite{GO}. Two-dimensional conformal field theory is deeply related to the theory of vertex operator algebras (see e.g. the introduction of \cite{FLM}).\\ 

The \textsl{Lie algebra of formal vector fields on a line} was defined 
by E. Cartan as early as 1909 \cite{Cartan}:
$$\Cal W_+(\K)=\Span_\K\Big\{ e_n= t^{n+1} \frac{\d}{\d t},\ n\geq -1 \Big\}.
\footnote{This Lie algebra is denoted $W_1(\K)$ by D.B. Fuks in \cite{Fuks}, 
who considers also the Lie algebras of formal vector fields on any 
$n$-dimensional space, $W_n(\K)$.}$$
When $\K=\CC$, $\Cal W_+(\CC)$ gives the vector fields on $\S^1$ which extend 
holomorphically to the unit disk. In his book, Fuks considers as well the 
following Lie subalgebras of $\Cal W_+(\K)$: 
$$L_k = \Span_\K\Big\{ e_n= t^{n+1} \frac{\d}{\d t},\ n\geq k \Big\},$$
satisfying the properties 
$$\Cal W_+(\K)=L_{-1} \supset L_0 \supset L_1 \supset\cdots, 
\qquad\mbox{and}\qquad [L_m,L_n]\subset L_{m+n} \quad\hbox{for $m,n\geq 0$}.$$

In particular, the Lie algebra $L_1$ is pronilpotent \cite{Cartier}. 
We show in the next section that $L_1$ is the Lie algebra of the group of 
formal diffeomorphisms $\Gdif(\K)$.

\subsection{Lie algebra of algebraic and proalgebraic groups}

Let $G_H$ be a proalgebraic group represented by a commutative Hopf algebra 
$H$ over a field $\K$. For any commutative unital $\K$-algebra $A$, an  
\textsl{$A$-valued infinitesimal character of $H$}, or \textsl{$A$-valued 
derivation on $H$} \cite{ConnesKreimer, Manchon08}, is a linear map $\alpha:H\to A$ such that, for any 
$x,y\in H$, the following relation holds:
$$\alpha(xy)=\alpha(x)e_A(y)+e_A(x)\alpha(y),$$
where $e_A=i_A\circ\varepsilon$. 
Since $e_A(1_H)=1_A$, this implies that $\alpha(1_H)=0$. 
Then, it is easily checked that the set ${\frak g}_H(A)=\Der(H,A)$ of 
$A$-valued derivations on $H$ is a Lie algebra, with bracket given by 
$$[\alpha,\beta] := \alpha\ast\beta-\beta\ast\alpha,$$ 
for any $\alpha,\beta\in {\frak g}_H(A)$. Moreover, the map 
$A\mapsto {\frak g}_H(A)$ is clearly a functor ${\frak g}_H$ from the category 
of commutative unital algebras into the category of Lie algebras. 
If $H$ is a finitely generated algebra (resp. inductive limit of finitely 
generated algebras), the functor ${\frak g}_H$ is by definition the Lie 
algebra of the algebraic (resp. proalgebraic) group $G_H$. 

\begin{prop}
Let $G_H$ be a prounipotent group, with Lie algebra ${\frak g}_H$, and suppose 
that the base field $\K$ is of characteristic zero. For any commutative unital 
algebra $A$, the exponential map 
\begin{eqnarray*}
\exp:{\frak g}_H(A) &\longrightarrow& G_H(A)\\
\alpha&\longmapsto& \exp (\alpha):= e_A +\sum_{n\ge 1}\frac{\alpha^{\ast n}}{n!}
\end{eqnarray*}
is a bijection, and we have $\exp(\alpha)\ast\exp(\beta)=\exp(\alpha+\beta)$ 
if $[\alpha,\beta]=0$. The inverse of the exponential map is the logarithm, 
defined by 
\begin{eqnarray*}
\log:G_H(A) &\longrightarrow& {\frak g}_H(A)\\
\varphi=e_A +\gamma &\longmapsto& 
\log (\varphi)=\sum_{n\ge 1}\frac{(-1)^{n-1}\gamma^{\ast n}}{n}. 
\end{eqnarray*}
\end{prop}

\begin{proof}
Let $m$ be a positive integer. For any $x\in H^m$ we have:
$$\Delta(x)=x\otimes 1_H+1_H\otimes x+\sum_{(x)}x'\otimes x'',$$
where $x',x"\in H^{m-1}$. As a consequence, for any linear map $\alpha:H\to A$ 
such that $\alpha(1_H)=0$, we have $\alpha^{\ast n}(x)=0$ for any $n\ge m+1$. 
This applies in particular when $\alpha$ is an infinitesimal character. 
Hence the exponential of $\alpha$ is well-defined, as the sum which defines it 
is locally finite. The same argument applies to the logarithm. The fact that 
the exponential of an infinitesimal character is a character is checked by 
direct inspection, as well as the last assertion. Finally let us consider any 
character $\varphi\in G_H(A)$. The powers $\varphi^{\ast m}$ are also characters 
for any positive integer $m$. Now let us define for any $\lambda\in \K$:
\begin{equation*}
\varphi^{\ast\lambda}:=\exp\big(\lambda\log(\varphi)\big).
\end{equation*}
For any $x,y\in H$, the expression 
$\varphi^{\ast\lambda}(x)\varphi^{\ast\lambda}(y)-\varphi^{*\lambda}(xy)$ is polynomial 
in $\lambda$ and vanishes at any positive integer, hence vanishes identically. 
It follows that $\varphi^{\ast\lambda}$ is a character for any $\lambda\in \K$. 
Differentiating with respect to $\lambda$ at $\lambda=0$ immediately gives the 
infinitesimal character equation for $\log (\varphi)$. A standard direct 
computation then shows that the logarithm and the exponential are mutually 
inverse. 
\end{proof}

Suppose now that the proalgebraic group $G_H$ is represented by a graded Hopf 
algebra $H=\oplus_{n\geq 0} H_n$, such that each graded component $H_n$ is 
finite-dimensional. Then one can define the  \textsl{graded dual Hopf algebra} 
$H^*:=\bigoplus_n H_n^*$ with the operations given in section \ref{ssec:Hopf}. 
Since $H$ is commutative, its dual $H^*$ is a cocommutative Hopf algebra. 
By the Cartier-Milnor-Moore theorem, then, $H^*$ is the universal enveloping 
algebra of the Lie algebra formed by its primitive elements, 
$$
\Prim H^*=\big\{ \alpha:H \longrightarrow \K \quad\mbox{linear},\quad 
m^*(\alpha)=\alpha\otimes \varepsilon + \varepsilon\otimes \alpha \big\}. 
$$
Applying the definition of $m^*$ one sees immediately that such primitive 
elements are $\K$-valued infinitesimal characters on $H$, and that  
$\Prim H^*$ is a Lie subalgebra of the Lie algebra ${\frak g}_H(\K)$. 
In fact, ${\frak g}_H(\K)$ is the completion of $\Prim H^*$. 

\subsubsection{Lie algebra for the Fa\`a di Bruno Hopf algebra}
\label{witt-virasoro-fdb}

As we saw, the group of formal diffeomorphisms is a prounipotent group 
represented by the Fa\`a di Bruno Hopf algebra $\HFdB=\K[x_1,x_2,...]$. 
This algebra is the free commutative algebra (i.e. the symmetric algebra) 
on the vector space spanned by the graded variables $x_n$, with degree 
$|x_n|=n$. Therefore it is graded, and its graded dual space $\HFdB^*$ 
is the universal enveloping algebra of a Lie algebra that we now describe.\\

The dual Hopf algebra $\HFdB^*$ is the cofree graded cocommutative coalgebra (the 
so-called \textsl{symmetric coalgebra}) on the vector space spanned by the 
dual elements of the variables $x_n$, that we denote by $e_n$. To be more 
precise, $\HFdB^*$ is isomorphic, as a graded vector space, to the symmetric algebra $S(V)$ where $V$ is the span of linear maps 
$e_n:\HFdB \longrightarrow \K$ such that $e_n(x_m)=\delta_{nm}$. The grading is induced by the grading on $V$ given by $|e_n|=-n$.\\

Elements of $V$ are clearly infinitesimal characters of $\HFdB$, i.e. elements of 
${\frak g}_{\HFdB}(\K)$. 
The Hopf algebra structure on $\HFdB^*$ is induced by that of $\HFdB$: 
the free commutative product on $\HFdB$ gives the \textsl{unshuffle coproduct} \cite{Ree}
on $\HFdB^*$, defined by:
$$\Delta(e_{i_1}\cdots e_{i_k})=\sum_{I\amalg J=\{1,\ldots k\}}e_I\otimes e_J,$$
where $I$ and $J$ are complementary parts of $\{1,\ldots,k\}$ and $e_I$ (resp. $e_J$) stands for the symmetric product of $e_{i_j}$'s with $j\in I$ (resp. $j\in J$). The Fa\`a di Bruno coproduct on $\HFdB$ gives the 
(noncommutative) multiplication on $\HFdB^*$ (of course different from the symmetric product we used just above). It is easy to check that the generators $e_n$ 
are the only primitive elements with respect to the unshuffle coproduct, 
therefore 
$$
\Prim \HFdB^* = \Span_\K\{e_1,e_2,...\}. 
$$
On the other hand, the multiplication induced by $\DFdB$ on $\HFdB^*$ 
coincides with the convolution product when it is restricted to 
${\cal{L}}(H,\K)$. By direct computations (or see \cite{BFM} about
$\big(\HFdBnc\big)^{op}$), we then 
obtain, for any $p,q\in \N-\{0\}$:
\begin{equation}\label{witt}
[e_p,e_q]:=e_p*e_q-e_q*e_p=(q-p)e_{p+q}.
\end{equation}
This means that $L_1$ is the Lie algebra of the group $\Gdif(\K)^{op}$. 
\medskip 

Finally note that the convolution product restricted to the generators 
of $L_1$ can be written as 
$$e_p*e_q=e_p\rhd e_q+\varphi,$$ 
where $\varphi$ is any element of $L_1$ which, seen as a linear map on $\HFdB$,
vanishes on the $x_n$'s for any $n\in\N$, and where we set 
$$
e_p\rhd e_q:=(q+1)e_{p+q}.
$$ 
The binary operation $\rhd$ is a \textsl{left pre-Lie} product \cite{Gerstenhaber, Vinberg}, which means 
that it satisfies the condition 
\begin{equation}
e_p\rhd(e_q\rhd e_r)-(e_p\rhd e_q)\rhd e_r = 
e_q\rhd(e_p\rhd e_r)-(e_q\rhd e_p)\rhd e_r, 
\end{equation}
for any $p,q,r>0$. Such products are discussed in section 
\ref{sssect:HFdB-preLie}. 
Going further, there is even a family $(\rhd_\lambda)_{\lambda\in\K}$ of left 
pre-Lie products on $L_1$, given by $e_p\rhd_\lambda e_q=(p+\lambda)e_{p+q}$, 
whose antisymmetrization gives back the Lie bracket \ref{witt}. 
See L. Foissy's notes in the present volume for more details.

\subsubsection{Lie algebra for the Hopf algebra of rooted trees}\label{lie-rt}

The graded dual $\HRT^*$ of the Hopf algebra of rooted trees is called 
\textsl{Grossman-Larson Hopf algebra}, after the first authors who 
considered it \cite{GrossmanLarson}. Being a cocommutative Hopf algebra, 
it is the enveloping algebra of a Lie algebra that we now describe.\\

The Hopf algebra $\HRT$ is the free commutative algebra on the vector space 
spanned by rooted trees. It is a graded Hopf algebra, if we consider the trees 
to be graded by the number of their vertices. A basis of $\HRT$ is given 
by monomials of trees, i.e. forests $(s)$. Denote by $(\delta_s)$ the dual 
basis in $\HRT^*$. The correspondence $\delta:s\mapsto \delta_s$ extends 
linearly to a unique vector space isomorphism from $\HRT$ onto $\HRT^*$. 
For any tree $t$, the corresponding $\delta_t$ is a primitive element of 
$\HRT^*$. In particular, it is an infinitesimal character of $\HRT$, i.e. 
an element of ${\frak g}_{\HRT}(\K)$. The Lie algebra of the group $\GRT(\K)$ 
is then the vector space 
$$\mathrm{Lie}(\GRT(\K)) = \Span_\K\{\delta_t,\ t\in \Cal T \}$$
with Lie bracket 
\begin{equation*}
[\delta_t,\delta_u]:= \delta_t\ast \delta_u-\delta_u\ast \delta_t=
\delta_{t\curvearrowright u-u\curvearrowright t}, 
\end{equation*}
where $t\curvearrowright u$ is the tree obtained by grafting $t$ on $u$ 
according to the rule
\begin{equation*}
t\curvearrowright u=\sum_{T\in \Cal T} N'(t,u,T)\ T,
\end{equation*}
where $N'(t,u,T)$ is the number of partitions $\Cal V(T)=V\amalg W$ of the 
set of vertices of the tree $T$ into two subsets with $W<V$, and such that 
$T\restr{V}=v$ and $T\restr{W}=u$. Another normalization is often employed 
to describe this Lie algebra. 
Consider in $\HRT^*$ the normalized dual basis
$\wt\delta_t=\sigma(t)\delta_t$, where $\sigma(t)=|\mop{Aut}t|$ stands for the
symmetry factor of $t$. Then 
$$\mathrm{Lie}(\GRT(\K)) \cong \Span_\K(\wt\delta_t,\ t\in \Cal T )$$
with Lie bracket 
\begin{equation*}
[\wt\delta_t,\wt\delta_u]:= 
\wt\delta_t\ast \wt\delta_u-\wt\delta_u\ast \wt\delta_t=\wt\delta_{t\to u-u\to t},
\end{equation*}
given in terms of the grafting operation  
\begin{equation*}
t\to u=\sum_{T \in \Cal T} M'(t,u,T)\ T,
\end{equation*}
where the factor 
$M'(t,u,T)=\displaystyle\frac{\sigma(t)\sigma(u)}{\sigma(T)}N'(t,u,T)$
can be interpreted as the number of ways to graft the tree $t$ on the tree $u$
in order to get the tree $T$. Otherwise said, $t\to u$ is the sum of the 
trees obtained by grafting $t$ on $u$ at any vertex $v$:
\begin{equation*}
t\to u=\sum_{v\in\Cal V(u)}t\to_v u.
\end{equation*}

The interesting aspect of the last presentation of $\mathrm{Lie}(\GRT(\K))$ 
is that the binary operation $\to$ on trees is a left pre-Lie product, i.e., 
for any trees $s,t,u$, we have:
\begin{equation*}
s\to(t\to u)-(s\to t)\to u=t\to(s\to u)-(t\to s)\to u. 
\end{equation*}
This identity is easily checked by remarking that the left-hand side is 
obtained by summing up all the trees obtained by grafting $s$ and $t$ on two 
(distinct or equal) vertices of $u$. This procedure is obviously symmetric 
in $s$ and $t$. Moreover, it has been shown by F. Chapoton and M. Livernet 
\cite{ChapotonLivernet} that the vector space spanned by the rooted trees 
$\Cal T$, endowed with the binary product $\to$, is in fact the 
\textsl{free pre-Lie algebra with one generator}. 
\medskip 

Let us consider the pre-Lie algebra $(L_1,\rhd)$ with basis $(e_n)_{n\ge 1}$ 
introduced in Paragraph \ref{witt-virasoro-fdb}. Using the universal property of the 
free pre-Lie algebra, there is a unique pre-Lie algebra morphism 
$$\phi:(\Cal T,\to)\longrightarrow (\Cal W,\rhd)$$
such that $\phi(\racine)=e_1$. It is a fortiori a Lie algebra morphism,  
and it is obviously surjective: indeed $e_n$ is the image of 
$\frac{1}{n!} \ell_n$, where $\ell_n$ is the ladder tree with $n$ vertices. 
It extends to a surjective Hopf algebra morphism between the corresponding 
enveloping algebras. Finally, the transposed map gives rise to the Hopf 
algebra embedding (\ref{fdb-bck}), namely $\Psi:\HFdB\inj 6 \HRT$,  
of section \ref{sssect:bck}. 

\section{Combinatorial interpretation}
\label{Section6}
\subsection{Incidence Hopf algebras}
Incidence Hopf algebras are Hopf algebras built from suitable families of 
partially ordered sets. They have been elaborated by W. R. Schmitt in 1994 
\cite{Schmitt}, following the track opened by S. A. Joni and G.-C. Rota when 
they defined incidence algebras and coalgebras (\cite{Rota, JoniRota}, see also \cite{Stanley}). They form a large family 
of Hopf algebras, which includes those on trees and the Fa\`a di Bruno one. 
We quickly describe here the subfamily of ``standard reduced'' incidence 
Hopf algebras, which are always commutative.
\medskip 

A \textsl{poset} is a partially ordered set $P$, with order relation denoted 
by $\le$ For any $x,y\in P$, the \textsl{interval} $[x,y]$ is the subset of $P$ formed by the elements $z$ such that $x\le z\le y$. Let $\P$ be a family of finite posets $P$ such that there exists 
a unique minimal element $0_P$ and a unique maximal element $1_P$ in $P$ (hence $P$ coincides with the interval $P=[0_P,1_P]$). 
This family is called \textsl{interval closed} if for any poset $P\in\P$ 
and for any $x\le y\in P$, the interval $[x,y]$ is an element of $\P$. 
Let $\overline{\P}$ be the quotient $\P/\sim$, where $P\sim Q$ if and 
only if $P$ and $Q$ are isomorphic as posets\footnote{W. Schmitt allows more 
general equivalence relations, called \textsl{order-compatible relations}.}. 
The equivalence class of any poset $P\in\P$ is denoted by $\overline P$ 
(notation borrowed from \cite{Ehrenborg}). 
The \textsl{standard reduced incidence coalgebra} of the family of posets 
$\P$ is the $\K$-vector space freely generated by $\overline{\P}$, 
with coproduct given by 
\begin{equation*}
\Delta[P]=\sum_{x\in P}\overline{[0_P,x]}\otimes \overline{[x,1_P]}, 
\end{equation*}
and counit given by $\varepsilon(\overline{\{*\}})=1$ and 
$\varepsilon(\overline P)=0$ if $P$ contains two elements or more.
\medskip 

Given two posets $P$ and $Q$, the \textsl{direct product} $P\times Q$ is the 
set-theoretic cartesian product of the two posets, with partial order given 
by $(p,q)\leq (p',q')$ if and only if $p\leq p'$ and $q\leq q'$. 
A family of finite posets $\P$ is called \textsl{hereditary} if the 
product $P\times Q$ belongs to $\P$ whenever $P,Q\in \P$. 
The quotient $\overline{\P}$ is then a commutative semigroup generated by 
the set $\overline{\P_0}$ of classes of \textsl{indecomposable posets}, 
i.e. posets $R\in\P$ such that for any $P,Q\in\P$ of cardinality 
$\ge 2$, $P\times Q$ is not isomorphic to $R$. The commutativity comes from 
the obvious isomorphism $P\times Q\sim Q\times P$ for any $P,Q\in \P$. 
The unit element $\un$ is the class of any poset with only one element.

\begin{prop}{\cite[Theorem 4.1]{Schmitt}}
If $\P$ is a hereditary family of finite posets, the standard reduced 
coalgebra $\Cal H(\P)$ of $\P$ is a Hopf algebra.
\end{prop}

\begin{proof}
The semigroup product extends bilinearly to an associative product on 
$\Cal H(\P)$. The compatibility of the coproduct with this product is 
obvious. The unit is a coalgebra morphism and the counit is an algebra 
morphism. The existence of the antipode comes from the fact that for any 
poset $P\in\P$ of cardinal, say, $n$, we obviously have:
$$
\Delta(\overline P)=\overline P\otimes\un+\un\otimes \overline P
   +\sum_{(\overline P)}\overline{P'}\otimes\overline{P"},
$$
where $P'$ and $P"$ contain strictly less than $n$ elements (note that the 
fact that $P$ is the interval $[0_P,1_P]$ is crucial here). Considering the reduced coproduct $\wt\Delta(\overline P)=\Delta(\overline P)-\overline P\otimes\un-\un\otimes \overline P$, the iterated reduced
coproduct $\wt{\Delta}^m(\overline P)$ therefore vanishes for $m>n$. 
It is well-known (see e.g. \cite{Manchon08}) that this conilpotence property 
allows us to define the convolution inverse of the identity, and even of any 
linear map $\varphi:\Cal H(\P)\to\Cal H(\P)$ with $\varphi(\un)=\un$.
\end{proof}

Many of the Hopf algebras encountered so far are incidence Hopf algebras. 
We give three examples, all of them borrowed from \cite{Schmitt}.

\subsubsection{The binomial and the divided power Hopf algebras}
Let $\Cal B$ be the family of finite boolean algebras. An element of $\Cal B$ 
is any poset isomorphic to the set $\P(A)$ of all subsets of a finite set 
$A$. The partial order on $\P(A)$ is given by the inclusion. 
If $B$ and $C$ are two subsets of $A$ with $B\subset C$, the interval $[B,C]$ 
in $\P(A)$ is isomorphic to $\P(C\backslash B)$, hence $\Cal B$ is 
interval-closed. Moreover the obvious property:
\begin{equation}
\P(A)\times\P(B)\sim \P(A\amalg B)
\end{equation}
implies that $\Cal B$ is hereditary. The incidence Hopf algebra 
$\Cal H(\Cal B)$ is the so-called \textsl{binomial Hopf algebra}, because 
of the expression of the coproduct on generators. In fact, as a vector space 
$\Cal H(\Cal B)$ is clearly spanned by $\{x_0,x_1,x_2,\ldots \}$, where $x_n$ 
stands for the isomorphism class of $\P(\{1,\ldots,n\})$. The product is 
obviously given by $x_mx_n=x_{m+n}$, the unit is $\un=x_0$, the counit is given 
by $\varepsilon(x_n)=0$ for $n\ge 1$, and the coproduct is entirely defined by 
$\Delta(x_1)=x_1\otimes \un+\un\otimes x_1$. Explicitly we have:
$$\Delta(x_n)=\sum_{k=0}^n{n\choose k}x_k\otimes x_{n-k}.$$

The graded dual Hopf algebra $\Cal H(\Cal B)^*$ is known as the 
\textsl{divided power algebra}: it can be represented as the vector space 
$\K\{y_0,y_1,y_2,\ldots \}$ with multiplication 
$$y_m\star y_n = {m+n\choose m}\ y_{m+n}$$
and unit $\un=y_0$. The counit is $\varepsilon(y_n)=0$ if $n>0$, and the 
coproduct is given by 
$$\D(y_n) = \sum_{p+q=n} y_p\otimes y_q.$$ 

\subsubsection{The Fa\`a di Bruno Hopf algebra}
Let $\Cal {SP}$ be the family of posets isomorphic to the set $\Cal {SP}(A)$ 
of all partitions of some nonempty finite set $A$. The partial order on set partitions 
is given by refinement. We denote by $0_A$ or $0$ the partition by singletons, 
and by $1_A$ or $1$ the partition with only one block. Let $\Cal Q$ be the 
family of posets isomorphic to the cartesian product of a finite number of 
elements in $\Cal {SP}$. If $S$ and $T$ are two partitions of a finite set $A$ 
with $S\le T$ (i.e. $S$ is finer than $T$), the partition $S$ restricts to a 
partition of any block of $T$. Denoting by $W/S$ the set of those blocks of 
$S$ which are included in some block $W$ of $T$, any partition $U$ such that 
$S\le U\le T$ yields a partition of the set $W/S$ for any block $W$ of $T$. 
This in turn yields the following obvious poset isomorphism:
\begin{equation}\label{partitions}
[S,T]\sim\prod_{W\in A/T} \Cal {SP}(W/S).
\end{equation}
This shows that $\Cal Q$ is interval closed (and hereditary by definition).

\begin{prop}{\cite[Example 14.1]{Schmitt}}
The incidence Hopf algebra $\Cal H(\Cal Q)$ is isomorphic to the Fa\`a di 
Bruno Hopf algebra.
\end{prop}

\begin{proof}
Denote by $X_n$ the isomorphism class of $\Cal{SP}(\{1,\ldots, n+1\})$. Note that $X_0$ is the unit of the Hopf algebra. In view of \eqref{partitions}, we have:
\begin{eqnarray}
\Delta(X_n) &=& \sum_{S\in\Cal{SP}(\{1,\ldots, n+1\})} 
   \overline{[0,S]}\otimes\overline{[S,1]}\nonumber\\
&=& \sum_{S\in\Cal{SP}(\{1,\ldots, n+1\})} 
   \left(\prod_{W\in \{1,\ldots , n+1\}/S}\overline{\Cal{SP}(W)}\right) \otimes 
   \overline{\Cal{SP}(\{1,\ldots, n+1\}/S)}\label{fdbsp}.
\end{eqnarray}
The coefficient in front of $X_1^{k_1}\cdots X_n^{k_n}\otimes X_m$ in \eqref{fdbsp} is equal to the number of partitions of $\{1,\cdots ,n+1\}$ with $k_j$ blocks of size $j+1$ (for $j=1$ to $n$), $m+1$ blocks altogether, and $k_0=m+1-k_1-\cdots -k_n$ blocks of size $1$. By definition of the Bell polynomials in terms of partitions, we have then:
\begin{equation}
\Delta(X_n)=\sum_{m=0}^n B_{m+1,n+1}(X_0,X_1,X_2,\ldots)\otimes X_m,
\end{equation}
 which in turn gives \eqref{doubilet} with $x_j:=\frac{1}{(j+1)!}X_j$.
\end{proof}

\subsubsection{The Hopf algebra of rooted trees}
Let $P$ be a finite poset not assumed tobe isomorphic to an interval. As an example, we can take as poset $P$ the vertex set 
$\Cal V(F)$ of a rooted forest $F$, in which $v\le w$ if and only if there is 
a path from one root to $w$ through $v$. An \textsl{order ideal} (or 
\textsl{initial segment}) in $P$ is a subset $I$ of $P$ such that for any 
$w\in I$, if $v\le w$, then $v\in I$. For a rooted forest $F$, an initial 
segment in $\Cal V(F)$ is a subforest such that any connected component of it 
contains a root of $F$. For any finite poset $P$, we denote by $J(P)$ the 
poset of all initial segments of $P$, ordered by inclusion 
\cite[Paragraph 16]{Schmitt}. 
The minimal element $0_{J(P)}$ is the empty set, and the maximal element 
$1_{J(P)}$ is $P$. For two finite posets $P$ and $Q$ one obviously has:
\begin{equation}\label{disj-union}
J(P\amalg Q)\sim J(P)\times J(Q).
\end{equation}
The isomrphism class of a poset $P$ is uniquely determined by the isomorphism 
class of the poset $J(P)$. To see this, consider two posets $P$ and $Q$, 
and suppose there is an isomorphism $\Phi:J(P)\to J(Q)$. For any $x\in P$, 
consider the initial segment $P_{\le x}:=\{y\in P,\  y\le x\}$. 
It has $x$ as unique maximal element. Now, $\Phi(P_{\le x})$ has a unique 
maximal element which we denote by $\varphi(x)$, and it is not hard to see 
that the map $\varphi:P\to Q$ thus constructed is a poset isomorphism.
\medskip 

For a poset $P$ and two initial segments $I_1\subset I_2\subset P$ with $I_1$ fixed, the 
correspondence $I_2\mapsto I_2\backslash I_1$ defines a poset isomorphism:
\begin{equation}\label{iso-initial}
[I_1,I_2]\subset J(P) \longrightarrow J(I_2\backslash I_1).
\end{equation}
Differences $Q=I_2\backslash I_1$ are \textsl{convex subsets} of $P$, i.e. 
such that for any $x,y\in Q$, we have $[x,y]\le Q$. Conversely, any convex 
subset $Q\subset P$ can be written as a difference $P_{\le Q}\backslash P_{<Q}$ 
of two unique initial segments:
\begin{eqnarray*}
P_{\le Q}&:=&\{x\in P,\ \exists y\in Q,\ x\le y\},\\
P_{< Q}&:=&\{x\in P,\ \forall y\in Q,\ x< y\}.
\end{eqnarray*}
Now let $\Cal F$ be a family of finite posets which is closed by disjoint 
unions and such that for any poset $P\in\Cal F$, convex subsets of $P$ also 
belong to $\Cal F$. Then the corresponding family:
$$J(\Cal F):=\{J(P),\, P\in\Cal F\}$$
is hereditary by virtue of isomorphisms \eqref{disj-union} and 
\eqref{iso-initial}.

\begin{prop}
The family $\Cal F$ of rooted forests is stable by taking disjoint unions and 
convex subposets, and the associated incidence Hopf algebra 
$\Cal H\big(J(\Cal F)\big)$ is isomorphic to $\HRT$.
\end{prop}

\begin{proof}
Via the isomorphism $\Phi$ defined above, the vector space 
$\Cal H\big(J(\Cal F)\big)$ can be identified with the vector space freely 
generated be the rooted forests. By \eqref{disj-union}, the product is then 
given by disjoint union, and the coproduct writes:
$$\Delta(P)=\sum_{I\in J(P)}(P\backslash I)\otimes I ,$$
which is just the coproduct \eqref{coprod} modulo flipping the terms 
(we have denoted by the same letter a forest and its underlying poset). 
The counit is given by $\varepsilon(\un)=1$ and $\varepsilon(P)=0$ for any 
nontrivial forest $P$. 
\end{proof}
\medskip 

This example can be extended, still in the context of incidence Hopf algebras, 
to oriented cycle-free graphs \cite{Manchon}. Hopf algebras of Feynman graphs 
do not enter strictly speaking into this framework (see however 
\cite[Paragraph 14]{Schmitt} about Hopf algebras of simple 3-connected graphs).
For other examples of incidence coalgebras and Hopf algebras, see 
\cite{JoniRota,Schmitt}.


\subsection{Combinatorial Hopf algebras}

The Fa\`a di Bruno Hopf algebra has some important properties also satisfied 
by almost all examples of graded Hopf algebras we have encountered so far 
(see \cite{JoniRota}). These properties are:
\begin{itemize}
\item
to be connected,
\item 
to be free or free commutative (or, dually, cofree or cofree cocommutative),
\item
to have an explicit basis,
\item
to have the ``right-sided property'':
\begin{equation}\label{rsp}
\D(x_n)=\sum\mbox{(polynomial in $x_k$)} \otimes  x_m 
\end{equation}
or its left analogue. 
\end{itemize} 
\smallskip
Hopf algebras of this type have been called \textsl{combinatorial Hopf 
algebras}\footnote{A different definition of combinatorial Hopf algebras was  
proposed, for commutative ones, by F. Hivert, J.-C. Novelli and J.-Y. Thibon 
in \cite{HNT}.} by J.-L. Loday and M. Ronco in \cite{LR2}. 
In particular, any combinatorial Hopf algebra fits into one of the four types 
below:
\bigskip 

{
\hskip 3cm
\framebox[3cm]{\quad \begin{minipage}{2.8cm}\smallskip 

$\big(T^a(V),\,\otimes,\,\D$?) \\ 
\medskip 

$\big(S^a(V),\,\cdot,\,\D?\big)$
\smallskip 
\end{minipage}}
\quad
\parbox{4.2cm}{$\underset{\mbox{graded Hopf algebra duality}}{\Longleftrightarrow}$}
\quad
\framebox[3cm]{\quad \begin{minipage}{2.8cm}\smallskip 
$\big(T^c(V),\,\D_{\otimes},\,\ast?\big)$ \\ 
\medskip 

$\big(S^c(V),\,\D_{\shu},\, \ast? \big)$
\smallskip 
\end{minipage}}
\bigskip 

\hspace{3.7cm} free type \hskip6.6cm cofree type
}
\bigskip\goodbreak

Let us explain the diagram above: here, $V$ is a graded vector space with finite-dimensional homogeneous components, $T^a(V)$ stands for the free algebra 
on $V$ endowed with the tensor product $\otimes$ (also called 
\textsl{concatenation} on words), and $T^c(V)$ stands for the graded cofree coalgebra 
on $V$ endowed with the \textsl{deconcatenation coproduct} $\D_{\otimes}$ 
defined by:
$$\D_{\otimes}(v_1\otimes\cdots\otimes v_n) = \sum_{k=0}^n 
(v_1\otimes\cdots\otimes v_k)\bigotimes(v_{k+1}\otimes\cdots\otimes v_n).$$
It can be seen as the graded dual of $T^a(V^*)$, where $V^*$ is the graded dual of $V$. In the same spirit, $S^a(V)$ stands for the free commutative algebra on $V$ 
endowed with the symmetrized tensor product $\cdot$, and $S^c(V)$ stands 
for the the graded cofree cocommutative coalgebra endowed with the 
unshuffle coproduct defined by: 
$$\D_{\shu}(v)=v\otimes 1+1\otimes v$$ 
for any $v\in V$, and extended multiplicatively (see \eqref{unshuffle} below). 
The operations $\D?$ and $*?$ stand for a priori unknown coproducts and 
products which the give rise to Hopf algebras of the desired type. 
\medskip 

J.-L. Loday and M. Ronco classified combinatorial Hopf algebras in \cite{LR2}. 
The simplest examples are the coordinate rings of the pro-algebraic groups 
$\Ginv$ and $\Gdif$ and their graded duals, which are of the symmetric type, 
and their non-symmetric lifts. 
In particular, their results imply the existence of a brace bracket on 
the enveloping algebras of $L_1$ and $\Cal W_+$, which reproduces the 
pre-Lie bracket already mentioned in section \ref{witt-virasoro-fdb} on the 
Lie algebras. We briefly discuss these examples. 

\subsubsection{The non-commutative Fa\`a di Bruno Hopf algebra and the brace 
bracket}
\label{sssect:brace}
The simplest examples of combinatorial Hopf algebras are the non-commuta\-tive 
lifts of the coordinate rings of the pro-algebraic groups $\Ginv$ and $\Gdif$, 
and their graded duals. In the case of the group $\Ginv$, we obtain two 
combinatorial Hopf algebras: 
\begin{align*}
& T^a(V)=\Hinvnc=\K\langle x_1,x_2,x_3,...\rangle, 
\quad\mbox{with product $\otimes$ and coproduct $\Dinv$}, \\ 
& T^c(V)=(\Hinvnc)^*=\K\langle x_1,x_2,x_3,...\rangle, 
\quad\mbox{with coproduct $^t\otimes=\D_{\otimes}$ and product $\ast=\Dinv^*$}.  
\end{align*}
We have already encountered the Hopf algebra $\Hinvnc$ in section 
\ref{ssect:HFdBnc}: it is endowed with the same coproduct $\Dinv$ as $\Hinv$, 
and it is therefore cocommutative. 
Dually, the operation $\ast=\Dinv^*$ is the commutative product: 
$$x_n\ast x_m = x_{n+m},$$
and $(\Hinvnc)^*$ is then isomorphic to the tensor coalgebra on the polynomial 
ring $\K[T]$ in one variable $T=x_1$, with isomorphism given by 
$x_n\mapsto T^n$. In the case of the group $\Gdif$, we obtain two combinatorial 
Hopf algebras: 
\begin{align*}
& T^a(V)=\HFdBnc=\K\langle x_1,x_2,x_3,...\rangle, 
\quad\mbox{with product $\otimes$ and coproduct $\DFdB$}, \\ 
& T^c(V)=(\HFdBnc)^*=\K\langle e_1,e_2,e_3,...\rangle, 
\quad\mbox{with coproduct $^t\otimes=\D_{\otimes}$ and product $\ast=\DFdB^*$}. 
\end{align*}
The Hopf algebra $\HFdBnc$, and its mysterious nature, has already been 
discussed in sections \ref{ssect:HFdBnc} and \ref{ssect:HFdBnc-question}. 
For its dual Hopf algebra $(\HFdBnc)^*$, we shall refer to the general 
classification theorem established by Loday and Ronco: if $H\cong T^c(V)$ 
for some graded vector space $V$, with cofree coassociative coproduct 
$\D_\otimes$ and product $\ast$ satisfying the dual of the right-sided 
property \eqref{rsp}, then \ignore{$H$ is a \textsl{dendriform algebra}  
\cite[Proposition 3.7]{LR2} and }$V$ is a \textsl{(right) brace algebra} \cite{LM}, 
i.e. is equipped with a bilinear map $\{\ ;\ \}:V\otimes T(V)\longrightarrow V$, 
called \textsl{(right) brace product}, such that
\begin{align}
\label{brace}
\{\{x;y_1\cdots y_p\};z_1\cdots z_q\} &
= \sum \{x;z_1\cdots  z_{i_1}\{y_1;z_{i_1+1}\cdots z_{j_1}\} 
z_{j_1+1}\cdots z_{i_p}\{y_p;z_{i_p+1}\cdots z_{j_p}\}z_{j_p+1}\cdots z_q\},
\end{align}
where the sum runs over all sequences 
$0\le i_1\le j_1\le\cdots\le i_p\le j_p\le q$. 
In the case $T^c(V)=(\HFdBnc)^*$, the right-sided property is fulfilled by 
the opposite algebra $(\HFdBncop)^*$, with product ${\DFdBop}^*$: in this case 
the brace product on the vector space $V=\Span_\K\{x_1,x_2,x_3,...\}$ has been 
computed in \cite{BFM}. 
If we omit the symbol $\otimes$ of the tensor product in the monomials, 
the brace product reads: 
$$
\{x_n;x_{m_1}\cdots x_{m_q}\} = \binom{n+1}{q} x_{n+m_1 +\cdots+m_q}.  
$$
For the algebra $T^c(V)=(\HFdBnc)^*$, which satisfies the left-sided property, 
one gets a \textsl{left brace product} 
$\{\ ;\ \}:T(V)\otimes V\longrightarrow V$ satisfying the obvious analogue 
properties of \eqref{brace}. 
\ignore{The dendriform structure $\ast$ on $(\HFdBncop)^*$ or $(\HFdBnc)^*$ can then be 
reconstructed from the brace bracket with a recursive formula \cite{LR2,BFM}.}

\subsubsection{The Fa\`a di Bruno Hopf algebra and the pre-Lie bracket}
\label{sssect:HFdB-preLie}

The symmetric versions of the two previous examples of combinatorial Hopf 
algebras are the coordinate rings of the pro-algebraic groups $\Ginv$ and 
$\Gdif$ (or $\Gdifop$), and their graded duals. 
In the case of the group $\Ginv$, we get the two combinatorial Hopf algebras
\begin{align*}
& S^a(V)=\Hinv=\K[x_1,x_2,x_3,...], \quad\mbox{with free commutative product 
$\cdot$ and coproduct $\Dinv$}, \\ 
& S^c(V)=\Hinv^*=\K\langle x_1,x_2,x_3,...\rangle, 
\quad\mbox{with coproduct $\D_{\shu}$ and product $\ast=\Dinv^*$}. 
\end{align*}
The unshuffle coproduct is defined on a monomial 
$x_{n_1}\cdots x_{n_k}\in \K\langle x_1,x_2,x_3,...\rangle$ (where we omit the 
tensor product $\otimes$ for simplicity) in the following way: 
\begin{align}
\label{unshuffle}
\D_{\shu}(x_{n_1}\cdots x_{n_k}) & = \sum_{p+q=k} \sum_{\sigma\in Sh(p,q)} 
x_{n_{\sigma(1)}}\cdots x_{n_{\sigma(p)}} \otimes x_{n_{\sigma(p+1)}}\cdots x_{n_{\sigma(p+q)}}, 
\end{align}
where $Sh(p,q)$ is the set of \textsl{$(p,q)$-shuffles}, i.e. the permutations 
$\sigma$ of $p+q$ numbers such that 
$$
\sigma(1)<\cdots<\sigma(p) \qquad\mbox{and}\qquad 
\sigma(p+1)<\cdots<\sigma(p+q).
$$
Both Hopf algebras are commutative and cocommutative: $\Hinv$ is the 
coordinate ring of the abelian group $\Ginv(\K)$, and $\Hinv^*$ is the 
enveloping algebra of the Lie algebra $\mathrm{Lie}\big(\Ginv(\K)\big)$, 
which is in fact the free abelian Lie algebra spanned by the variables 
$x_1,x_2,x_3,...$, and is obviously isomorphic to the polynomial ring 
$K[T]$ in one variable. 
\bigskip 

In the case of the group $\Gdif$, we get the two combinatorial Hopf algebras
\begin{align*}
& S^a(V)=\HFdB=\K[x_1,x_2,x_3,...], \quad\mbox{with free commutative product 
$\cdot$ and coproduct $\DFdB$}, \\ 
& S^c(V)=\HFdB^*=\K\langle e_1,e_2,e_3,...\rangle, 
\quad\mbox{with coproduct $\D_{\shu}$ and product $\ast=\DFdB^*$}. 
\end{align*}
In the symmetric case, the Loday-Ronco classification theorem (studied before 
also by T. Lada, M. Markl \cite{LM} and by D. Guin, J.-M. Oudom \cite{GO}) 
says the following: if $H\cong S^c(V)$ is endowed with the unshuffle 
coproduct $\D_{\shu}$, and if the product $\ast$ satisfies the dual of the 
right-sided property \eqref{rsp}, then $V$ is a \textsl{right pre-Lie 
algebra}, i.e. is equipped with a binary product 
$\lhd:V\otimes V\longrightarrow V$ such that:
\begin{align}
\label{right-pre-Lie}
(x\lhd y)\lhd z-x\lhd(y\lhd z)=(x\lhd z)\lhd y-x\lhd(z\lhd y).
\end{align}
Comparing with the results in the non-symmetric case, a pre-Lie product 
is an example of a brace product 
$\{\ ;\ \}:V\otimes T(V)\longrightarrow V$, which vanishes outside the 
subspace $V\otimes V$, that is 
$$\{x;y\}= x\lhd y \qquad\mbox{and}\qquad \{x,y_1\cdots y_m\}=0 
\quad\mbox{if $m\neq 1$}.$$ 
On the other side, any pre-Lie product on $V$ gives rise to a 
\textsl{symmetric brace product} $\{\ ;\ \}:V\otimes S(V)\longrightarrow V$, 
recursively defined by 
$$
\{x;1\}=x, \qquad \{x;y\}=x\lhd y, \qquad\mbox{and}\qquad \{x;y_1\cdots y_p\}=
\{x;y_1\cdots y_{p-1}\}\lhd y_{p}-\{x;\{y_1;y_2\cdots y_{p-1}\}\}. 
$$ 
Details can be found in \cite{GO} and \cite{LM}.\\

This general result applies in particular to the Hopf algebra 
$\HFdB^*=U(L_1)$ (modulo a switch between right and left properties), and tells 
that the Lie algebra of vector fields $L_1$ is a symmetric left brace algebra, 
and thus a left pre-Lie algebra \cite{BFM}: the opposite of the brace bracket 
given in \eqref{brace} is in fact symmetric, and induces on $L_1$ the left 
pre-Lie product 
$$
e_n\lhd e_m=\{e_n;e_m\}^{op}=(m+1)\ e_{n+m} 
$$ 
that we saw in section \ref{witt-virasoro-fdb}. 
\bigskip 

Note that the pre-Lie bracket defined on $L_1$ can be extended to a left 
pre-Lie bracket on the Lie algebra $\Vect(\S^1)=\Span_{\K}\{e_n,\ n\in \Z\}$. 

\subsubsection{The shuffle Hopf algebra}
A last example of a combinatorial Hopf algebra is the \textsl{shuffle Hopf algebra} 
$T^c(V)= \K\langle x_1,x_2,...\rangle$ endowed with the deconcatenation 
$\D_\otimes$ and the \textsl{shuffle product} 
$$
(x_{n_1}\otimes\cdots\otimes x_{n_p})\shu(x_{n_{p+1}}\otimes\cdots\otimes x_{n_{p+q}}) 
= \sum_{\sigma\in Sh(p,q)} x_{n_{\sigma(1)}}\otimes\cdots\otimes x_{n_{\sigma(p+q)}}. 
$$
We give the degree $n$ to the letter $x_n$. This commutative graded Hopf algebra represents the prounipotent group scheme $A\mapsto G_A$, whose Lie algebra is $\frak g_A=\frak g\otimes A$, where $\frak g$ is the free Lie algebra over the alphabet $\{x_1,x_2,\ldots\}$.
\section{Operadic interpretation}
{sect:operad}
Operads are combinatorial devices which appeared in algebraic topology, coined for coding ``types of algebras''. Hence, 
for example, a Lie algebra is an algebra over some operad denoted by $Lie$, 
an associative algebra is an algebra over some operad denoted by $Assoc$, 
a commutative algebra is an algebra over some operad denoted by $Com$, etc. The term "operad" has been invented by J.-P. May \cite{May}, although the notion itself appears a few years before in an article by J. M. Boardman and R. M. Vogt \cite{BV}. For a comprehensive and recent textbook on operads, we refer the reader to \cite{LV}. See also \cite{Loday96} for a more concise overview.
\subsection{Manipulating algebraic operations}

\noindent
For any $n\in \N$, an $n$-ary operation on a vector space $V$ is an 
$n$-multilinear map 
$$
a:V^{\otimes n}\longrightarrow V\ ,\ (v_1,...,v_n)\mapsto a(v_1,...,v_n)
$$
suitably represented by a box:
$$ 
\operation
$$  
where the $n$ inputs represent the variables $v_1,...,v_n$ and the output 
represents the result $a(v_1,...,v_n)$ of the operation. 
For $n=1$ there is a canonical operation, namely the identity map 
$\mop{Id}:V\to V$. One can even consider $0$-ary operations, as being just 
distinguished elements of $V$.\\

Given an $n$-ary operation $a$, one can consider new $n$-ary operations 
by permuting the inputs of $a$, that is, by setting 
$$
a^{\sigma}(v_1,...,v_n) = a(v_{\sigma(1)},...,v_{\sigma(n)}),
$$
for any permutation $\sigma\in S_n$. 
For instance, given a binary operation $a(v_1,v_2)=v_1\ast v_2$, one can 
consider the opposite operation 
$$
v_1\ast^{op} v_2= v_2\ast v_1 = a^{(12)}(v_1,v_2). 
$$ 
Given a ternary operation $b(v_1,v_2,v_3)$, one can define algtogether six operations 
$b_{\sigma}$ simply by permuting the inputs with a permutation $\sigma\in S_3$. 
In all cases, the symmetric group $S_n$ obviously acts (from the right) 
on the set of $n$-ary operations, by permuting their entries.\\

Given any two operations $a$ and $b$ on $V$, one also obtains new operations 
by applying them one after the other. For instance, if $a$ is an $n$-ary 
operation, and $b$ is an $m$-ary operation, we get an $(n+m-1)$-ary operation 
by applying first $b$ on any subset of $m$ variables among the $n+m-1$ 
available, and then applying $a$ to the $n$ variables thus obtained. 
The operation thus obtained is the \textsl{partial composition} 
$a\circ_i b$, for any choice of $i=1,...,n$ where the result of $b$ is 
inserted as a variable of $a$, namely 
$$
(a\circ_i b)(v_1,...,v_{n+m-1}) := 
a(v_1,\cdots,v_{i-1},b(v_i,....,b_{i+m-1}),v_{i+m},...,v_{n+m-1}). 
$$
Graphically: 
$$
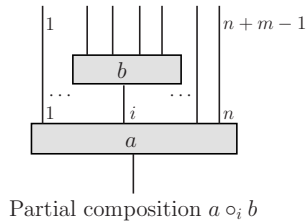
\operade
$$  
For instance, if we compose a binary operation $a(v_1,v_2)=v_1\ast v_2$ 
with itself, we obtain two ternary operations 
$$
(v_1,v_2,v_3)\mapsto (v_1\ast v_2)\ast v_3,\qquad\mbox{and}\qquad 
(v_1,v_2,v_3)\mapsto v_1\ast (v_2\ast v_3)
$$
which do not change the order of the inputs, and twelve operations 
if we allow the permutations of the entries. 
\medskip 

The ``types of algebras'' that one may wish to consider are characterized 
by the properties of the operations that one can perform. The concept of 
operad emerges when one tries to write such properties only in terms of the 
operations, discarding the entries. For example, a vector space $V$ is an 
associative algebra if it is endowed with a binary operation 
$a:V^{\otimes 2}\to V$, denoted by $\ast$, which is associative, 
i.e. for any $x,y,z\in V$ we have: 
$$
(x\ast y)\ast z = x\ast (y\ast z) \qquad\mbox{that is}\qquad 
a \circ_1 a = a \circ_2 a. 
$$
Similarly, a vector space $V$ is a Lie algebra if it is endowed with a binary 
operation $a:V^{\otimes 2}\to V$, denoted by $[\ ,\ ]$, which is antisymmetric 
and satisfies the Jacobi identity, i.e. for any $x,y,z\in V$ we have:  
\begin{eqnarray*}
[x,y]+[y,x]&=&0  \hbox{ that is } a + a^{(12)} = 0,\\ 
\,\![[x,y],z]+[[y,z],x]+[[z,x],y]&=&0 \hbox{ that is } 
a\circ_1 a + (a\circ_1 a)^{\tau} + (a\circ_1 a)^{\tau^2} =0, 
\end{eqnarray*}
where $\tau$ is the circular permutation $(123)\in S_3$. 

\subsection{Algebraic operads}
\label{operads}

An \textsl{algebraic operad}, or \textsl{linear operad}, is a collection 
$\P=\big(\P(n)\big)_{n\ge 0}$ of $\K$-vector spaces $\P(n)$ 
on which the symmetric group $S_n$ acts from the right, together with a 
distinguished element $e\in\P(1)$, called the \textsl{identity}, 
and a collection of linear maps, called \textsl{partial compositions}, 
\begin{align*}
\circ_i:\P(k)\otimes\P(l) &\longrightarrow \P(k+l-1),\hskip 8mm
i=1,\ldots ,k\\
(a,b)&\longmapsto a\circ_i b
\end{align*}
satisfying the following associativity, unit and equivariance axioms: 
\begin{itemize}
\item
The \textsl{nested} and the \textsl{disjoint associativity}: 
for any $a\in \P(k)$, $b\in\P(l)$, $c\in\P(m)$ one has:
\begin{align}
(a\circ_i b)\circ_{i+j-1} c &= a\circ_i(b\circ_j c),
\qquad i\in\{1,\ldots,k\}, \quad j\in \{1,\ldots,l\},\\
(a\circ_i b)\circ_{l+j-1}c &= (a\circ_j c)\circ_i b, 
\qquad i,j\in\{1,\ldots,k\},\quad i<j.
\end{align}
Graphically, these compositions produce the following operations: 
\begin{equation*}
\asoperade
\end{equation*}
\item
The \textsl{unit property}:
\begin{align}
e\circ a&=a\\
a\circ_ie&=a,\qquad i=1,\ldots ,k. 
\end{align}
\item
The \textsl{equivariance property}:
\begin{equation}
a^\sigma\circ_{\sigma(i)} b^\tau=(a\circ_i b)^{\eta_i(\sigma,\tau)}
\end{equation}
where $\eta_i(\sigma,\tau)\in S_{k+l-1}$ is defined by letting $\tau$ permute
the set $E_i=\{i,i+1,\ldots,i+l-1\}$ of cardinality $l$, and then by letting
$\sigma$ permute the set $\{1,\ldots,i-1,E_i,i+l,\ldots,k+l-1\}$ of 
cardinality $k$. 
\end{itemize}
\bigskip 

The prototype of algebraic operads is the \textsl{endomorphism operad} 
on a vector space $V$, denoted by $\mop{End}(V)$. For any $n\geq 0$, 
it is given by the vector space of $n$-multilinear maps on $V$, that is,
\begin{equation*}
\mop{End}(V)(n)=\Cal L(V^{\otimes n},V).
\end{equation*}
The right action of the symmetric group $S_n$ on $\mop{End}(V)(n)$ is 
induced by the left action of $S_n$ on $V^{\otimes n}$ given by:
\begin{equation*}
\sigma.(v_1\otimes\cdots\otimes v_n) = 
v_{\sigma^{-1}_1}\otimes\cdots\otimes v_{\sigma^{-1}_n}.
\end{equation*}
The unit element is the identity map $e:V\to V$, and the partial compositions 
are given by true compositions of linear maps.\\

Given an operad $\P$, a \textsl{$\P$-algebra} is a vector space $A$ 
together with a morphism of operads from $\P$ to $\mop{End}(A)$, 
that is, for any $n\geq 0$, an equivariant linear map 
$$
\P(n) \longrightarrow \mop{End}(A)(n)=\Cal L(A^{\otimes n},A) 
$$
which identifies each element $a$ of $\P(n)$ with an $n$-ary operation 
on $A$, and which preserves the partial compositions. 
Given two $\P$-algebras $A$ and $B$, a \textsl{morphism of 
$\P$-algebras} from $A$ to $B$ is a linear map $f:A\to B$ such that, 
for any $n\ge 0$ and for any $a\in\P(n)$, the following diagram commutes,
\diagramme{
\xymatrix{A^{\otimes n} \ar[rr]^{f^{\otimes n}}\ar[d]^{a} & & B^{\otimes n}\ar[d]^{a}\\ 
A \ar[rr]_f & & B }
}
\noindent where we have denoted by the same letter $a$ the element of 
$\P(n)$ and its images in $\mop{End}(A)(n)$ and $\mop{End}(B)(n)$.\\

\noindent
Using the partial compositions one can also define a \textsl{total
composition}
\begin{align*}
\gamma:{\P}(n)\otimes{\P}(k_1)\otimes\cdots\otimes{\P}(k_n)&
\longrightarrow {\P}(k_1+\cdots +k_n) \\
(a,b_1,\ldots,b_n) & \longmapsto 
\gamma(a;b_1,\ldots ,b_n)=\Big(...\big((a\circ_n b_n)\circ_{n-1} b_{n-1}\big)\cdots\Big)\circ_1 b_1, 
\end{align*}
which is graphically represented as follows:
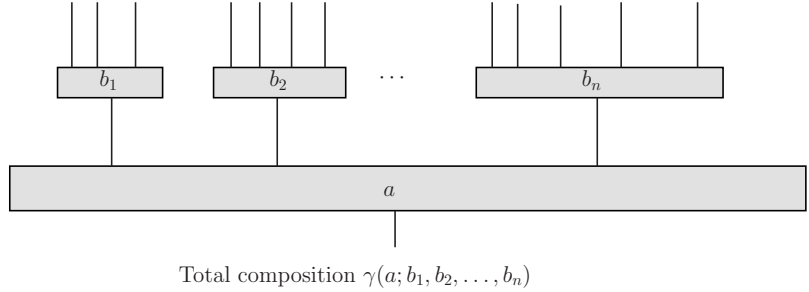
\begin{equation*}
\gcoperade
\end{equation*}
Such a map allows us to regard an operad $\P$ as a \textsl{Schur functor} \cite{Macdonald, LV}
on $\K$-vector spaces: 
$$
\P: \Vect_{\K}\longrightarrow \Vect_{\K},\ 
V \longmapsto \P(V)= \bigoplus_{n\geq 0} \P(n)\otimes_{S_n} V^{\otimes n}
$$
endowed with an associative composition 
$\gamma: \P\circ \P\longrightarrow \P$ and with a unit 
$i:\Cal I \hookrightarrow \P$ from the trivial operad 
$\Cal I(n)=\delta_{n,1} \K e$. 
A $\cal P$-algebra can then be also seen as a vector space $A$ endowed 
with an associative linear map $\gamma_A:\P\circ\P(A)\longrightarrow \P (A)$,  
i.e. such that $\gamma_A(\gamma_A\circ I)=\gamma_A(I\circ \gamma_A):\P\circ\P\circ \P (A)\longrightarrow\P(A)$. The usual type of algebras give rise to the following operads: 
\begin{itemize}
\item 
$Assoc$ for associative algebras: $Assoc(n)=\K[S_n]$.

\item 
$Com$ for associative and commutative algebras: $Com(n)$ is given by the 
trivial representation of $S_n$.

\item
$Lie$ for Lie algebras: the vector space $Lie(n)$ is the representation 
of $S_n$ induced from the one-dimensional representation $\rho$ of $C_n=\Z/\Z_n$ given by a 
primitive $n$-th root of unity, i.e. $Lie(n)=\Ind_{C_n}^{S_n}(\rho)$.

\item 
$preLie$ for right pre-Lie algebras \cite{Gerstenhaber,Vinberg}: 
as a vector space $preLie(n)=\Span_{\K}\{\Cal T_n\}$ is spanned by the rooted 
trees with $n$ labeled vertices, and the symmetric groups $S_n$ acts by 
permuting the labels \cite{ChapotonLivernet}.
\end{itemize}
For more details about operads, see e.g. \cite{Loday96,LV}.


\subsection{Pre-Lie algebras, Lie algebras and groups associated to operads}\label{pl-op}
For operads $\P$ such that $\P(0)=\{0\}$, there is a canonical way to 
construct a pronilpotent Lie algebra and, dually, a prounipotent group. 
For simplicity, let us illustrate this construction for regular operads. 
\bigskip 

An operad $\P$ is \textsl{regular} \cite[Paragraph 5.2.9]{LV} if for any 
$n\ge 0$ we have $\P(n)=\P_n\otimes \K[S_n]$, where $\P_n$ is some vector space. In this case, the operad 
$\P$ gives rise to the \textsl{non-symmetric operad} 
$\P'=(\P_n)_{n\geq 0}$, which verifies axioms similar to those of an operad, 
except that the actions of symmetric groups are not required. Any operad 
is also a non-symmetric operad, if we forget the symmetric group actions. 
An algebra over a non-symmetric operad $\P'$ is a vector space $A$ 
together with a non-symmetric operad morphism from $\P'$ into the 
underlying non-symmetric operad of $\mop{End}(A)$.\\

An operad $\P$ is regular if the relations defining the $\P$-algebras 
can be written without switching the arguments. For example, the associativity 
identity $(x\ast y)\ast z=x\ast (y\ast z)$ does not change the order of the 
variables: thus $Assoc$ is regular. Indeed $Assoc(n)=\K[S_n]$, and the 
associated non-symmetric operad $Assoc'$ is given by $Assoc'_n=\K$. 
The operads $Com$, $Lie$ and $preLie$ are not regular.\\

Let $\P$ be a regular operad, hence $\P(n)=\P_n\otimes \K[S_n]$. 
The \textsl{Lie algebra associated to $\P$} was introduced by M. Kapranov 
in 2000, see \cite{KM}, as the vector space 
$$L^{\P} = \bigoplus_{n=2}^\infty\ \P_n= 
\left\{ \sum_{\mbox{\footnotesize finite}}\ p_n\ \Big|\ p_n\in \P_n \right\}$$
with Lie bracket $[p,q]= p\lhd q-q\lhd p$ induced by the right pre-Lie product 
$$p\lhd q=\sum\ \g(q;\id,...,p,...,\id),$$
where the sum runs on the possible places to put $p$. 
Since $\P_n \lhd \P_m\subset \P_{n+m-1}$, the above pre-Lie and Lie products 
are graded by $|p|=n-1$ for $p\in\P_n$. As the sum starts at $n=2$, this Lie algebra is pronilpotent. The completion $\widehat{L^{\P}}$ 
of the Lie algebra $L^{\P}$, with respect to this graduation, contains formal 
series, i.e. possibly infinite sums $\sum_{n\geq 2} p_n$, with $p_n\in\P_n$.
\medskip 

If the operad $\P$ is not regular, a similar construction makes sense if 
we consider the sum of coinvariant spaces, with respect to the action of 
the symmetric group, that is, if we set:
$$
L^{\P}= \bigoplus_{n=2}^\infty \P(n)/S_n.
$$
Pre-Lie and Lie products are defined similarly, as the sum of all partial 
compositions still makes sense modulo the symmetric group actions.
\bigskip  

Let $\P$ be a regular operad which satisfies
${\P}(0)=\{0\}$ and ${\P}(1)=\K e$. 
The \textsl{group associated to $\P$} was defined independently by 
F. Chapoton \cite{Chapoton} and P. van der Laan \cite{VanDerLaan} as the set 
$$G^{\P}= \left\{ \sum_{n=1}^\infty\ p_n\ \Big|\ 
p_n\in \P_n \hbox{ and } p_1=e \right\},$$ 
with the composition law:
$$\sum_{n\ge 1} p_n \circ \sum_{m\ge 1}\ q_m = 
\sum_{n\ge 1}\sum_{m_1,\ldots m_n\ge 1} \g(p_n;q_{m_1},...,q_{m_n}).$$ 
A similar construction is again possible for non-regular operads, 
if we replace $\P_n$ by $\P(n)/S_n$. The group $G^{(\P)}$ is prounipotent, 
with associated pronilpotent Lie algebra $\widehat{L^{\P}}$.
\medskip 

Two Hopf algebras can then be associated to $\P$: the co-commutative Hopf 
algebra $U(L^{\P})$, and its graded dual $H^{\P}$, which is commutative 
and represents the proalgebraic group $G^{\P}$, i.e. 
$G^{\P}(A)=\mop{Hom}_{\smop{CAlg}}(H^\P,A)$ for any unital and commutative 
algebra $A$.

\subsubsection{The Fa\`a di Bruno Hopf algebra and the associative operad}

F. Chapoton and P. van der Laan proved \cite{Chapoton,VanDerLaan} that 
the group $G^{\P}$ is the group of formal diffeomorphisms $\Gdif$ if $\P$ is the 
non-symmetric operad $Assoc$ of associative algebras. In this case, 
in fact, in each degree $n\geq 1$ there is only one possible operation, 
up to a scalar factor and permutation. If we choose a generator $p_n$ for each vector space $Assoc'_n$ (with the notation introduced in the beginning of Section \ref{pl-op}), an element of the group $G^{\P}$ is a formal series 
$$
f= \sum_{n=1}^\infty\ f_n\ p_n, \qquad\mbox{with $f_n\in\K$ and $f_1=1$}, 
$$
and the operadic composition of two such series gives exactly the 
Fa\`a di Bruno composition of diffeomorphisms.  If
$\P$ is the (non-regular) operad $Com$ of commutative algebras, the construction above gives back the group $\Ginv$ \cite{Chapoton}.\\

A similar construction of the group $\Gdif$ was done in \cite{Frabetti} 
for set-operads. In this case, there was proven a general criterion to compare 
the group $\Gdif$ with the group $G^{\P}$ associated to suitable operads. An operad $\P$ is a \textsl{set-like operad} if the relations defining 
the operations do not involve sums or multiplications by scalars. 
For instance, the operads $Assoc$ and $Com$ are set-like. 
A set-like operad $\P$ is then the collection of the vector spaces $\P(n)$ 
generated by some sets $\P_{\smop{set}}(n)$. The collection 
$\P_{\smop{set}}=\big(\P_{\smop{set}}(n)\big)_{n\ge 1}$ gives rise to a 
\textsl{set operad}, that is, an operad in the category of sets. 
The axioms for set operads are the same as for linear operads, except that 
the tensor product of vector spaces is replaced by the cartesian product 
of sets, and the direct sum is replaced by the disjoint union. 
If moreover the set-like operad is regular, the associated set operad is 
regular as well, i.e. $\P_{\smop{set}}(n)=\P_{\smop{set},n}\times S_n$ for any 
$n\ge 1$.\\

\noindent
If $\P$ is a regular set-like operad, the proalgebraic group $G^\P$ was called the 
\textsl{group of $\P$-expanded series} in \cite{Frabetti}:
$$G^{\P}(A) = 
\big\{ f(x)=\sum_{p\in\P_n}\ f_p\ x^p\ \big|\ f_p\in A \big\},$$ 
where $x^p$ is just a symbol, with product given by a formal composition law
$$(f\circ g)(x) = \sum_{p\in\P}\  \sum_{q_1,...,q_{|p|}\in \P}\ 
f_p\ g_{q_1}\cdots g_{q_{|p|}}\ x^{\g(p,q_1,...,q_{|p|})}.$$ 
\bigskip 

The simplest example of a regular set-like operad is $Assoc$, with one single $n$-ary operation up to permutation for any $n\ge 1$ which is denoted by $n-1$. Hence the non-symmetric set 
operad $Assoc'_{\smop{set}}$ can be identified with the set $\N$ of 
(non-negative) natural numbers. The operadic partial composition of two 
operations $p$ and $q$ (which is unique by virtue of the associativity) 
coincides with the sum $p+q$. Similarly, the total composition 
$\gamma(n;m_0,m_1,...,m_n)$ coincides with the sum $m_0+\cdots + m_n$. 
For any unital commutative algebra $A$, the group $G^{Assoc'_{\smop{set}}}(A)$ is 
again the group of formal diffeomorphisms $\Gdif(A)$. In \cite{Frabetti} it was shown also that, for any regular set operad $\P$, 
the order map  
$$\P\longrightarrow Assoc'_{\smop{set}},\quad p\mapsto |p|$$ 
induces a canonical surjective morphism of groups 
$G^{\P}(A) \to\!\!\!\!\!\to \Gdif(A)$.  
Moreover, any element $p_2\in \P_2$ which satisfies the associativity axiom, 
if it exists, gives an operad morphism 
$$
Assoc'_{\smop{set}} \longrightarrow \P, 
\quad n \longmapsto p_n=\gamma(p_2;p_{n-1},e)
$$
which induces a section $\Gdif(A)\longrightarrow G^{\P}(A)$. 

\subsubsection{The QED charge Hopf algebra on planar binary trees and the 
duplicial operad}
Another example of a regular set operad is the operad $Dup$ which defines 
\textsl{duplicial algebras} \cite{Z}, that is, algebras $A$ endowed with two 
binary operations \textsl{over} $\over$ and \textsl{under} $\under$, 
satisfying the following identities: 
\begin{align*}
(x\over y)\over z &= x\over (y\over z)  \\ 
(x\over y)\under z &= x\over (y\under z) \\ 
(x\under y)\under z &= x\under (y\under z), 
\end{align*}
for any $x,y,z\in A$. The operad Dup is described by means of 
\textsl{planar binary trees} \cite{Frabetti,LodayTriple}, that is, 
planar rooted trees with internal vertices of valence $3$. 
For any $n\geq 0$, $Dup_n$ is the set of trees with $n$ internal vertices:
\begin{align*}
& Dup_0 = \left\{ \treeO\right\}, \qquad Dup_1= \left\{ \treeA \right\}, 
\qquad Dup_2 = \left\{ \treeAB, \treeBA \right\}, \\ 
& Dup_3 = \left\{ \treeABC, \treeBAC, \treeACA, \treeCAB, \treeCBA \right\}. 
\end{align*}
The operations over and under are defined on any two trees $t$ and $s$ 
as the following grafting of the root of a tree onto the left-most 
or onto the right-most leaf of the other one: 
\begin{align*}
\mbox{$t$~{\em over}~$s$:}\qquad & t \over s =   \lgraft{s}{t} \\ 
\mbox{$t$~{\em under}~$s$:}\qquad & t \under s =  \rgraft{t}{s}
\end{align*}
The operadic composition in Dup is then given by the following rule: 
for any trees $t$ and $s_1,...,s_{|t|}$, the tree 
$\gamma(t;s_1,...,s_{|t|})$ is obtained by replacing each internal vertex 
of $t$ by the trees $s_1,...,s_{|t|}$, in the order given by the decomposition 
of $t$ as a monomial in the vertex tree $\raise 2pt\hbox{$\treeA$}$, 
using the over and under products and suitable parentheses. 
For instance, 
\begin{align*}
 &\treeBAC = (\treeA \under \treeA) \over \treeA  
\quad\mbox{hence}\quad 
\mu_{\treeBAC}(s_1,s_2,s_3) = (s_1 \under s_2) \over s_3 
= \lgraft{s_3}{\rgraft{s_1}{s_2}}. 
\end{align*}

The group $G^{Dup}(A)$ of \textsl{tree-expanded series} of the form 
$$
f(x)=\sum_{t\in Dup} f_t\ x^t, 
\qquad\mbox{with $f_t\in A$ and $f_Y=1$}
$$
was studied extensively in \cite{Frabetti}. Since the operad $Dup$ 
contains a binary associative operation $\raise 2pt\hbox{$\treeA$}$, 
the group $G^{Dup}(A)$ projects onto the group $\Gdif(A)$ and 
at the same time contains a copy of $\Gdif(A)$. 
Moreover, this proalgebraic group is represented by a commutative 
Hopf algebra $H_{Dup}$ on planar binary trees which admits a non-commutative 
lift different from the Loday-Ronco Hopf algebra on planar binary trees 
\cite{LR}, and different from its linear dual.\\

Finally, the group $G^{Dup}(\CC)$ contains a proalgebraic subgroup 
$G^\alpha(\CC)$ whose representative Hopf algebra $H^{\alpha}$ was used 
in \cite{BFqedren,BFqedtree} to describe the renormalization of the 
electric charge in massless quantum electrodynamics. The algebra $H^{\alpha}$ 
is a suitable quotient of the algebra $H_{Dup}$, such that the group 
$G^\alpha(\CC)$ can be seen as the set of tree-expanded series of the form 
$$
\alpha_f(x)= \Big(x^{\treeO}-x^{\treeA}\under f(x)\Big)^{-1} \over x^{\treeA}, 
$$
for any tree-expanded series $f(x)=\sum f_t\ x^t$. Here the inversion 
of the series $x^{\treeO}-x^{\treeA}\under f(x)$ is performed in the group 
of \textsl{tree-expanded invertible series}, which is the set of tree-expanded 
series starting with the term $x^{\treeO}$, and considered with the product
\textsl{over}.



\begin{thebibliography}{100}
\addcontentsline{toc}{section}{\bf References}
\bibitem{Arbogast}
L. F. A. Arbogast,
\textsl{Du calcul des d\'erivations}, Levrault, Strasbourg (1800).

\bibitem{Agarwala}
S. Agarwala,
\textsl{The Geometry of Renormalization},  Ph.D. Thesis, Johns Hopkins University (2008).

\bibitem{Bell}
E. T. Bell,
{\em Exponential polynomials\/},
Ann. of Math. (2) {\bf 35} (1934), 258--277.

\bibitem{BS}
M. Bellon, F. Schaposnik, 
\textsl{Renormalization group functions for the Wess-Zumino model: 
up to 200 loops through Hopf algebras }, 
Nucl. Phys. \textbf{A800} (2008), 517--526.

\bibitem{BH}
G. M. Bergman, A. O. Hausknecht,
\textsl{Co-groups and co-rings in categories of associative rings},
Math. Surveys and Monographs \textbf{45}, Amer. Math. Soc. (1996).

\bibitem{Berstein}
I. Berstein,
\textsl{On co-groups in the category of graded algebras},
Trans. Amer. Math. Soc. \textbf{115} (1965), 257--269.

\bibitem{Block}
R. E. Block, 
{\em On the Mills-Seligman axioms for Lie algebras of classical type}, 
Trans. Amer. Math. Soc., \textbf{121} (1966), 378--392

\bibitem{BV}
J. M. Boardman, R. M. Vogt,
\textsl{Homotopy-everything $H$-spaces},
Bull. Amer. Math. Soc. \textbf{74} No 6 (1968), 1117--1122.

\bibitem{Brouder} 
C.~Brouder, 
{\em On the trees of quan\-tum fields}, 
Eur.\ Phys.\ J.~C {\bf 12} (2000), 535--549. 

\bibitem{BFqedren} 
C.~Brouder and A.~Frabetti, 
{\em Renormalization of QED with planar binary trees}, 
Eur.\ Phys.\ J.~C {\bf 19} (2001), 715--741. 

\bibitem{BFqedtree} 
C.~Brouder and A.~Frabetti, 
{\em QED Hopf algebras on planar binary trees}, 
J.~Alg.\ {\bf 267} (2003), 298--322.

\bibitem{BFM}
C. Brouder, A. Frabetti, F. Menous,
\textsl{Combinatorial Hopf algebras from renormalization},
J. Alg. Combinatorics \textbf{32} (2010), 557--578.


\bibitem{Butcher63}
J. Butcher,
\textsl{Coefficients for the study of Runge-Kutta integration processes},
J. Austral. Math. Soc. \textbf{3} No2, 185--201 (1963).

\bibitem{Butcher}
J. Butcher,
\textsl{An algebraic theory of integration methods},
Math. Comp. \textbf{26} (1972), 79--106.

\bibitem{BFK} 
C.~Brouder, A.~Frabetti and C.~Krattenthaler,  
{\em Non-commutative Hopf algebra of formal diffeomorphisms}, 
Adv.~Math.\ {\bf 200} (2006), 479--524. 

\bibitem{Cartan}
E. Cartan,
\textsl{Les groupes de transformations continus, infinis, simples},
Ann. Sci. Ecole Norm. Sup. \textbf{26} (3) (1909), 93--161.

\bibitem{Cartier}
P. Cartier,
\textsl{A primer of Hopf algebras},
Frontiers in Number Theory, Physics and Geometry II, 537--615, Springer (2007).

\bibitem{Chapoton}
F. Chapoton,
\textsl{Rooted trees and exponential-likes series},
preprint, \texttt{arXiv:math/0209104}

\bibitem{ChapotonLivernet}
F. Chapoton, M. Livernet,
\textsl{Pre-Lie algebras and the rooted trees operad},
Int. Math. Res. Notices \textbf{8} (2001), 395--408.

\bibitem{ConnesKreimer} 
A.~Connes and D.~Kreimer, 
{\em Hopf Algebras, Renormalization and Noncommutative Geometry.} 
Comm.\ Math.\ Phys.\ {\bf 199} (1998) 203--242.

\bibitem{ConnesKreimerI} 
A.~Connes and D.~Kreimer, 
{\em Renormalization in quantum field theory and the Riemann--Hilbert 
problem.~I. The Hopf algebra structure of graphs and the main theorem.} 
Comm.\ Math.\ Phys.\ {\bf 210} (2000), 249--273. 

\bibitem{ConnesKreimerII} 
A.~Connes and D.~Kreimer, 
{\em Renormalization in quantum field theory and the Riemann--Hilbert 
problem.~II. The $\beta$-function, diffeomorphisms and the renormalization 
group.} 
Comm.\ Math.\ Phys.\ {\bf 216} (2001), 215--241.

\bibitem{CM}
A. Connes, H. Moscovici,
\textsl{Hopf algebras, cyclic cohomology and the transverse index theorem},
Comm. Math. Phys. \textbf{198} (1998), 199--248.

\bibitem{Craik}
A. D. D. Craik,
\textsl{Prehistory of Fa\`a di Bruno's formula},
Amer. Math. Monthly \textbf{112} No2 (2005), 119--130.

\bibitem{Doubilet}
P. Doubilet,
\textsl{A Hopf algebra arising from the lattice of partitions of a set}, 
J. Algebra \textbf{28} (1986), 127--132.

\bibitem{Dur}
A. D\" ur,
\textsl{M\"obius functions, incidence algebras and power series representations},
Lect. Notes math. \textbf{1202}, Springer (1986).

\bibitem{Ebrahimi-FardGuo} 
K.~Ebrahimi-Fard and L.~Guo,  
{\em Rota-Baxter Algebras in Renormalization of Perturbative 
Quantum Field Theory\/}, 
Fields Institute Communications 50 (2007), 47--105.

\bibitem{EP1}
K. Ebrahimi-Fard, F. Patras, 
\textsl{Exponential renormalization}, Ann. Henri Poincar\'e \textbf{11} (5), (2010), 943--971.

\bibitem{EP2}
K. Ebrahimi-Fard, F. Patras,
\textsl{Exponential Renormalization II: Bogoliubov's R-operation and momentum subtraction schemes}, J. Math. Phys. \textbf{53}, 083505 (2012). \texttt{arXiv:1104.3415v1} (2011).

\bibitem{EH}
B. Eckmann, P. J. Hilton, \textsl{Group-like structures in general categories I, II and III}, Math. Ann. \textbf{145} (1962), 227-255, ibid. \textbf{151} (1963), 150-186 and ibid. \textbf{150} (1963), 165--187.

\bibitem{Ehrenborg}
R. Ehrenborg, \textsl{On posets and Hopf algebras},
Adv. Math. \textbf{119} (1996), 1--25.

\bibitem{FaadiBruno} 
F.~Fa\`a di Bruno, 
{\em Sullo sviluppo delle funzioni}, 
Ann. Sci. Mat. Fis., Roma {\bf 6} (1855), 479--480.

\bibitem{FGV}
H. Figueroa, J. Gracia-Bondia, J. Varilly, \textsl{Fa\`a di Bruno Hopf algebras}, \texttt{arXiv:math/0508337} (2005).

\bibitem{FGV2}
H. Figueroa, J. Gracia-Bondia, \textsl{Combinatorial Hopf algebras in quantum field theory I}, Rev. Math. Phys. \textbf{17}  (2005), 881--976.

\bibitem{F}
L. Foissy,
\textsl{Les alg\`ebres de Hopf des arbres enracin\'es d\'ecor\'es I,II},
Bull. Sci. Math. \textbf{126} (2002), 193--239 and 249--288.

\bibitem{F2}
L. Foissy,
\textsl{Fa\`a di Bruno subalgebras of the Hopf algebra of planar trees from combinatorial Dyson-Schwinger equations}, Adv. Math. \textbf{218} (2008), 136--162.

\bibitem{Frabetti} 
A.~Frabetti, 
{\em Groups of tree-expanded formal series},
Journal of Algebra 319 (2008), 377--413.

\bibitem{FLM}
I. Frenkel, J. Lepowsky, A. Meurmann,
\textsl{Vertex operator algebras and the Monster}
Pure and Appl. Math. \textbf{134}, Academic Press, New York (1988).

\bibitem{FruchtRota}
R. Frucht, G-C. Rota,
{\em Polinomios de Bell y particiones de conjuntos finitos\/},  
Scientia {\bf 126} (1965), 5--10.

\bibitem{Fuks}
D.~B. Fuks,
{\em Cohomology of infinite-dimensional Lie algebras\/},  
Consultants Bureau, New York 1986.

\bibitem{GelfandFuks}
 D.~B. Fuks, I. M. Gelfand,
{\em Cohomology of Lie algebras of vector fields on the circle\/},  
Funkts. Anal. Prilozhen.  \textbf{2} No. 4 (1968), 92--93. 

\bibitem{Gerstenhaber}
M. Gerstenhaber 
{\em The cohomology structure of an ass-ociative ring},  
 Ann. Math.  \textbf{78}, No2 (1963), 267-288. 

\bibitem{Gelfand}
    I.~M.~Gelfand, D.~Krob, A.~Lascoux, B.~Leclerc, V.~Retakh and J.-Y.~Thibon,
    ``Noncommutative symmetric functions'',
    \textit{Adv. Math.} 112 (1995), 218--348.
    
\bibitem{Giacardi}
L. Giacardi,
\textsl{Francesco Fa\`a di Bruno. Ricerca scientifica, insegnamento e 
divulgazione},
Deputazione subalpina di storia patria, Studi e fonti \textbf{XII}, 
Torino (2004).

\bibitem{GirelliKrajewskiMartinetti}
F.~Girelli, T.~Krajewski, P.~Martinetti, 
{\em An algebraic Birkhoff decomposition for the continuous 
renormalization group\/}, 
J.~Math.~Phys. {\bf 45} (2004), 4679--4697.

\bibitem{GO}
P. Goddard, D. Olive eds.,
\textsl{Kac-Moody and Virasoro algebras, a reprint volume for physicists},
Adv. Series in Math. Phys. \textbf{3}, World Scientific (1988).

\bibitem{GrossmanLarson}
R. Grossman, R. G. Larson, 
\textsl{Hopf-algebraic structure of families of trees},
J. Algebra \textbf{126}, No 1 (1989), 184--210.

\bibitem{GD}
W. S. Gray, L. A. Duffaut Espinosa,
\textsl{A Fa\`a di Bruno Hopf algebra for a group of Fliess operators with 
applications to feedback},
Systems \& control Lett. \textbf{60} (7) (2011), 441--449.

\bibitem{GO}
D. Guin, J.-M. Oudom, \textsl{On the Lie enveloping algebra of a pre-Lie 
algebra}, 
J. of K-theory: K-theory and its applications to algebra, geometry and 
topology \textbf{2} issue 01 (2008) 147--167.

\bibitem{HLW}
E. Hairer, Ch. Lubich, G. Wanner,
\textsl{Geometric numerical integration: structure-preserving algorithms for ordinary differential equations},
2nd edition, Springer (2006).

\bibitem{HNT}
F. Hivert, J.-C. Novelli, J.-Y. Thibon, 
\textsl{Commutative combinatorial Hopf algebras},
J. Alg. Combinatorics \textbf{28} No 1 (2008), 65--95.

\bibitem{Hochschild} 
G.~Hochschild, 
{\em La structure des groupes de Lie}, 
Dunod 1968. 

\bibitem{Johnson} 
W.~P.~Johnson, 
{\em The curious history of Fa\`a di Bruno's formula}, 
AMS Monthly {\bf 109} (2002), 217--234. 

\bibitem{JoniRota}
S. A. Joni, G.-C. Rota,
\textsl{Coalgebras and bialgebras in combinatorics},
Stud. Appl. Math. \textbf{61} No2 (1979), 145--163.

\bibitem{KacRaina}
V. Kac, A. Raina,
\textsl{Bombay lectures on highest weight representations of infinite-dimensional Lie algebras},
Adv. Series in Math. Phys. \textbf{2}, World Scientific (1987).

\bibitem{Kac}
V. Kac,
\textsl{Vertex algebras for beginners},
Univ. Lectures Series \textbf{10},
Amer. Math. Soc., Providence (1997).

\bibitem{KM}
M. Kapranov, Yu. Manin,
\textsl{Modules and Morita theorem for operads},
Amer. J. Math. \textbf{125}, No5 (2001), 811--838.

\bibitem{Knight}
T. Knight,
\textsl{On the expansion of any functions of multinomials},
Philos. Trans. R. Soc. London \textbf{101} (1811), 49--88.

\bibitem{Kreimer} 
D.~Kreimer, 
{\em On the Hopf algebra structure of perturbative quantum field theories}, 
Adv. Theor. Math. Phys.{\bf 2} (1998), 303--334.

\bibitem{Lacroix}
S. F. Lacroix,
\textsl{Trait\'e du calcul diff\'erentiel et du calcul int\'egral},
(3 volumes), Duprat, Paris, 1810--1819.

\bibitem{LM}
T. Lada, M. Markl, 
\textsl{Symmetric brace algebras}, 
Appl. Categorical Structures \textbf{13}, Issue 4 (2005), 351--370.

\bibitem{Lagrange}
J.-L. Lagrange, 
{\em Nouvelle m\'ethode pour r\'esoudre les \'equations litt\'erales 
par le moyen des s\'eries}, 
M\'emoires de l'Acad\'emie Royale des Sciences et Belles-Lettres de Berlin 
{\bf 24} (1770) 251--326.

\bibitem{LL}
J. Lepowsky, H. Li,
\textsl{Introduction to vertex operator algebras and their representations},
Progress in Math. \textbf{227}, Birkh\" auser, Boston (2003).

\bibitem{Loday96}
J-L. Loday,
\textsl{La renaissance des op\'erades},
S\'eminaire Bourbaki, Vol. 1994/95. Ast\'erisque \textbf{237}, 
Exp. No. 792, 3 (1996), 47--74.

\bibitem{LodayTriple}  
J.-L. Loday, 
\textsl{Generalized bialgebras and triples of operads},  
Ast\'erisque \textbf{320} (2008), 116 pages.

\bibitem{LR}
J.-L. Loday, M. Ronco,
\textsl{Hopf algebra of the planar binary trees},
Adv. Math. \textbf{139} (1998), 293--309.

\bibitem{LR2}
J.-L. Loday, M. Ronco,
\textsl{Combinatorial Hopf algebras},
Quanta of Math., Clay Math. Proc. \textbf{10} (2008), 347--383.

\bibitem{LV}J-L. Loday, B. Vallette, \textsl{Algebraic operads\/},
Grund. Math. Wiss. \textbf{346}, Springer (2012).
www-irma.u-strasbg.fr/$\sim$loday/

\bibitem{Macdonald} 
I.~G.~Macdonald, 
{\em Symmetric functions and Hall polynomials\/}, 
Oxford University Press 1979.

\bibitem{Maclane}
S. Mac Lane,
\textsl{Categories for the working mathematician},
Second edition, Springer (1998).

\bibitem{Manchon08}
D. Manchon
\textsl{Hopf algebras and renormalisation}, 
Handbook of algebra \textbf{5} (M. Hazewinkel ed.) (2008), 365--427. 

\bibitem{Manchon}
D. Manchon,
\textsl{On bialgebras and Hopf algebras of oriented graphs},
Confluentes Math. \textbf{4} No1 (2012), 10 pages, \texttt{arXiv:1011:3032}.

\bibitem{May}
J. P. May,
\textsl{The Geometry of Iterated Loop Spaces}, 
Springer-Verlag (1972).

\bibitem{Menous}
F. Menous,
\textsl{Formulas for the Connes-Moscovici Hopf algebra},
C. R. Math. Acad. Sci. Paris \textbf{341} (2) (2005), 75--78.


\bibitem{Ree}
R. Ree,
\textsl{Lie elements and the algebra associated with shuffles},
Ann. Math. \textbf{68} No2 (1958), 210--219.

\bibitem{Riordan}
J. Riordan,
{\em Derivatives of composite functions\/},
Bull. Amer. Math. Soc. \textbf{52} (1946), 664--667.

\bibitem{Rota}
G.-C. Rota, \textsl{On the foundations of combinatorial theory I: theory of M\" obius functions},
Z. Wahrscheinlichkeitstheorie and Werw. Gebiete \textbf{2} (1964), 340--368.

\bibitem{Schmitt}
W. R. Schmitt, 
{\em Incidence Hopf algebras\/},
Journal of Pure and Applied Algebra \textbf{96} (1994), 299--330. 

\bibitem{Stanley}
R.~P.~Stanley, 
{\em Enumerative combinatorics\/}, 
2nd. edition, Cambridge University Press (1997).

\bibitem{Sweedler}
M. E. Sweedler, \textsl{Hopf algebras},
Benjamin (1969).

\bibitem{T}
E. C. Titchmarsh, \textsl{Theory of Functions},
Oxford (1939).

\bibitem{VanDerLaan} 
P.~van der Laan, 
{\em Operads and the Hopf algebras of renormalization}, 
preprint (2003), 
{ http://www.arxiv.org/abs/math-ph/0311013}.

\bibitem{VanSuijlekom} 
W.~van Suijlekom, 
{\em Multiplicative renormalization and Hopf algebras}, 
preprint (2007), 
{ arXiv:0707.0555}

\bibitem{Vinberg}
E. B. Vinberg,
{\em The theory of convex homogeneous cones},
Trans. Amer. Math. Soc. \textbf{12} (1963), 340--403.

\bibitem{V87}
D. Voiculescu,
\textsl{Dual algebraic structures on operator algebras related to free products},
J. Operator Theory \textbf{17} (1987), 85--98.

\bibitem{Waterhouse}
W. C. Waterhouse,
\textsl{Introduction to affine group schemes},
Graduate Texts in Mathematics, Springer (1979).

\bibitem{Zassenhaus}
H. Zassenhaus, 
{\em \"Uber Lie’sche Ringe mit Primzahlcharakteristik}, 
Abhandlungen aus dem Mathematischen Seminar der Universit\"at Hamburg
\textbf{13} No. 1 (1939), 1--100.

\bibitem{Zhang} 
J.~Zhang, 
{\em $H$-algebras}, 
Adv. in Math. \textbf{89} (1991), 144--191.

\bibitem{Z}
G. W. Zinbiel,
\textsl{Encyclopedia of types of algebras 2010}, 
Proc. Int. Conf., in Nankai Series in Pure, Applied Mathematics and
Theoretical Physics \textbf{9} (World Scientific, Singapore, 2012), 217--298. 
\end{thebibliography}
\end{document}